\documentclass[abstract=true,10pt]{scrartcl}

\RequirePackage{amsthm,amsmath,amsfonts,amssymb}
\RequirePackage{graphicx}

\usepackage[left=2cm,right=2.5cm]{geometry}
\usepackage{authblk} %For Footnote Style Authorlist
\usepackage{epsfig}
\usepackage{color}
\usepackage{dsfont}
\usepackage{bm}
\usepackage{verbatim}
\usepackage{stmaryrd}
\usepackage{bbm}			
\usepackage{enumerate} 
\usepackage{nicefrac}
\usepackage{xfrac}
\usepackage{units} 
\usepackage{mathtools}
\usepackage{pdfcomment}
\usepackage{url}
\usepackage{soul}
\usepackage[labelfont={bf,sf},font={footnotesize},labelsep=space,format=plain]{caption}
\usepackage{acronym}
\usepackage{mathrsfs}
\usepackage[T1]{fontenc}
\usepackage{caption, booktabs}
\usepackage{textcomp}
\usepackage{multirow}
\usepackage{lipsum}
\usepackage{makecell}
\usepackage[section]{placeins}
\usepackage{hologo}
\usepackage{xcolor}
\usepackage{dutchcal}
\usepackage[title]{appendix}%

 \theoremstyle{plain}
 \newtheorem{theorem}{Theorem}[section]
 \newtheorem{Lemma}[theorem]{Lemma}
 \newtheorem{Cor}[theorem]{Corollary}

 \theoremstyle{definition}

 \newtheorem{Rem}[theorem]{Remark}
 \newtheorem{?}[theorem]{Problem}
 \newtheorem{Ex}[theorem]{Example}

\providecommand{\keywords}[1]
{
  \small	
  \textbf{\textit{Keywords---}} #1
}

\makeatletter
\renewcommand{\maketitle}{\bgroup\setlength{\parindent}{0pt}
\begin{center}
  \textbf{\@title}
\end{center}

\begin{flushleft}
	\@author
\end{flushleft}\egroup
}
\makeatother

\title{\begin{LARGE}On differentiability and mass distributions of topologically typical multivariate Archimedean copulas\end{LARGE}}

\date{}

\author[1,a]{Nicolas Dietrich}
\author[1,b]{Wolfgang Trutschnig}
\affil[1]{\begin{small} Department of Artificial Intelligence and Human Interfaces, University of Salzburg, Austria \end{small}}

\affil[a]{\begin{small}\url{nicolaspascal.dietrich@plus.ac.at}\end{small}}
\affil[b]{\begin{small}\url{wolfgang.trutschnig@plus.ac.at}\end{small}}

\begin{document}
\maketitle

\begin{abstract}
Copulas, in particular Archimedean copulas are commonly viewed as analytically nice and 
regular objects. Motivated by a recently established result sta\-ting that the first partial derivatives of 
bivariate copulas can exhibit surprisingly pathological behavior, we focus on the class of $d$-dimensional
Archimedean copulas denoted by $\mathcal{C}_{ar}^d$ and show that partial derivatives of order $(d-1)$ 
can be sur\-pri\-singly irregular as well. In fact, we prove the existence of 
Archimedean copulas $C \in \mathcal{C}_{ar}^d$ whose $(d-1)$-st order partial 
derivatives are pathological in the sense that for almost every $\mathbf{x} \in [0,1]^{d-1}$ the derivative 
$\partial_1...\partial_{d-1}C(\mathbf{x},y)$ does not exist on a dense set of $y \in (0,1)$. \\
Since the existence of mixed partial derivatives of order $(d-1)$ of a copula $C$ 
is closely related to the existence of a discrete component,  
we also study mass distributions of Archimedean copulas. 
Building upon the interplay between Archimedean copulas and so-called Williamson measures
we show that absolute continuity, discreteness and singularity 
of the Williamson measure propagates to the associated Archimedean copula and vice versa. 
Moreover, we prove the fact that the sub-family of $\mathcal{C}_{ar}^d$ consisting of copulas 
whose absolutely continuous, discrete and singular component have full support is dense in 
$\mathcal{C}_{ar}^d$.  \\
Finally, viewing $\mathcal{C}_{ar}^d$ in the light of Baire categories, we show that, in contrast to the space of bivariate copulas, 
a topologically typical $d$-dimensional Archimedean copula $C$ is not absolutely continuous but 
has degenerated discrete component, implying that pathological elements are rare in $\mathcal{C}_{ar}^d$.  
\end{abstract}


\section{Introduction}
Due to their simple algebraic structure and their relevance in various applications, e.g., in finance and hydrology 
\cite{durante2007, mult_arch, nelsen2006}, Archimedean copulas gained increasing popularity over the last decades. Considering a so-called Archimedean generator $\psi \colon [0,\infty) \rightarrow [0,1] =: \mathbb{I}$ and its 
quasi-inverse function $\varphi$ (and assuming that the generator $\psi$ is sufficiently monotone/regular), setting 
$$
C_\psi(x_1,...,x_d) := \psi(\varphi(x_1) + ... + \varphi(x_d)).
$$
defines a $d$-dimensional Archimedean copula $C_\psi$. 
Building upon the afore-mentioned regularity of the generator $\psi$ it is known 
(see, e.g., \cite[Theorem A1]{kasper2024}) that for an arbitrary $d$-dimensional Archimedean copula $C$ 
and every $y \in (0,1)$ the partial derivative $\partial_1...\partial_{d-1}C(\mathbf{x},y)$ exists for 
$\lambda_{d-1}$-almost every $\mathbf{x} := (x_1,...,x_{d-1}) \in \mathbb{I}^{d-1}$ (whereby $\lambda_{d-1}$ denotes the $(d-1)$-dimensional Lebesgue measure and $\mathbb{I}$ the unit interval). 
However, as pointed out in the bivariate setting in \cite{dietrich2024}, derivatives of copulas have to 
be handled with care - the main objective of this paper is to show that this statement remains true for the  
class $\mathcal{C}_{ar}^d$ of all $d$-dimensional Archimedean copulas. \\
Motivated by the results in \cite{dietrich2024}, where it was shown that in the whole family 
$\mathcal{C}^2$ of bivariate copulas as well as in the family of bivariate Extreme Value copulas 
there exist elements $C$ with the property that the first partial derivative $\partial_1C(\cdot,\cdot)$ 
does not exist on a set of full support, we here establish a similar result for the family of $d$-dimensional Archimedean copulas and show the existence of an Archimedean copula $C$ with the following property: 
there exists some Borel set $\Lambda \subseteq \mathbb{I}^{d-1}$ with $\mu_{C^{1:d-1}}(\Lambda) > 0$ 
(whereby $\mu_{C^{1:d-1}}$ is the $(d-1)$-stochastic measure corresponding to the marginal copula $C^{1:d-1}$ of the first $d-1$ coordinates) such that for $\lambda_{d-1}$-almost every $\mathbf{x} \in \mathbb{I}^{d-1}$ 
the mixed partial derivative $\partial_1...\partial_{d-1}C(\mathbf{x},y)$ does not exist for 
a dense set of $y \in (0,1)$. More surprisingly, the family of all copulas exhibiting the afore-mentioned pathological behavior is still comparably large in the sense that it is dense in $\mathcal{C}_{ar}^d$ with respect to 
the standard uniform metric $d_\infty$.\\
Building upon the fact that (for arbitrary dimension $d \geq 2$) multivariate Archimedean copulas $C$ 
can be characterized in terms of so-called Williamson measures $\gamma$ (probability measures $\gamma$ on $(0,\infty)$, 
see \cite{neslehova}), we first derive various results on the interplay 
between the copula and its Williamson measure, which provide the basis for constructing co\-pulas with
afore-mentioned pathological differentiability properties. 
In particular, we show that the Williamson measure $\gamma$ has a
non-degenerated absolutely continuous/discrete/singular (singular in the sense that it has no point masses and its corresponding distribution function $F_\gamma$ has derivative 0 $\lambda$-almost everywhere) component if, and only if 
the corresponding Archimedean copula $C$ has non-degenerated absolutely continuous/discrete/singular component (in 
a sense specified in Section \ref{section:arch_cop}). 
Moreover we prove that $C \in \mathcal{C}_{ar}^d$ has full support if, and only if the correspon\-ding Williamson 
measure $\gamma$ has full support and then show that the same holds for the absolutely continuous, the discrete 
and the singular component. Building upon these results we are able to prove the surprising fact 
that the sub-family of $\mathcal{C}_{ar}^d$ consisting of copulas 
whose absolutely continuous, discrete and singular component have full support is dense in 
$\mathcal{C}_{ar}^d$.  \\
One natural question that naturally arises is, 
whether the afore-mentioned subfamilies are considered topologically `small' or `large', indicating 
whether their elements represent atypical or typical elements of $\mathcal{C}_{ar}^d$.
Utilizing Baire categories (see, e.g., \cite{oxtoby1980}), topology provides a natural framework for distinguishing between `small' and `large' sets. A subset of a topological space $(\mathcal{T},\tau)$ is 
called \textit{nowhere dense} if, and only if the interior of the closure of that set is empty in $(\mathcal{T},\tau)$. Furthermore, a set in $(\mathcal{T},\tau)$ is called \textit{meager}/of \textit{first Baire category}, if it can be covered by a countable union of nowhere dense sets. A set is of \textit{second Baire category}, if it is not of first Baire category and, finally, a set is called \textit{co-meager}, if it is the complement of a set of first Baire category. Following \cite{bruck1997} and sticking to the concept of `small' and `large' sets, in a complete metric 
space, sets of first Baire category are the `small' sets, sets of second Baire category are referred to as 
`not small' and co-meager sets are the `large' sets. 
Moreover, given a topological space $(\mathcal{T},\tau)$, we call an element of a co-meager set 
a \textit{typical} element and an element of a meager set an \textit{atypical} element of that space.\\
In the family $\mathcal{C}^2$ of all bivariate copulas a typical copula is discrete 
(even mutually completely dependent, see \cite{dietrich2024,kim}), whereas a typical Extreme Value copula 
has degenerated discrete component (again see \cite{dietrich2024}). 
Following along these lines, using the afore-mentioned results on differentiability and mass distributions, 
as main result on sizes of subclasses of $\mathcal{C}_{ar}^d$ we finally show that 
a typical $d$-dimensional Archimedean copula has non-degenerated discrete component, is not absolutely 
continuous and has full support. As a direct consequence of this result we can provide a new, significantly 
simplified proof of \cite[Theorem 2.5]{cat_exchange}, stating that for dimension $d = 2$ a typical Archimedean copula is strict. For a selection of further contributions viewing copulas from the Baire category perspective
we refer to \cite{durante2022,cat_exchange,cat_quas_cop,typ_cop_sing} and the references therein.\\

The rest of this paper is organized as follows: Section \ref{section:notation} containing 
notations and definitions used throughout this paper is divided into two parts: the first subsection 
focuses on general notation used throughout this paper while the second recalls 
relevant properties of Archimedean copulas and their interplay with Williamson measures. 
Section \ref{section:regularity} studies regularity and mass distributions of Archimedean copulas. 
After showing how regularity and measure-theoretic properties of the Williamson measure propagate 
to the corresponding Archimedean copula (and vice versa) we prove that the family of Archimedean copulas 
whose discrete, absolutely continuous and singular component have full support $\mathbb{I}^d$ 
is dense in $(\mathcal{C}_{ar}^d,d_\infty)$. After deriving analogous results 
for the Kendall distribution function in Section \ref{sec:kendall}, in Section \ref{section:derivatives} 
we focus on $(d-1)$-st order partial derivatives $\partial_1...\partial_{d-1}C$ of Archimedean copulas $C$ 
and establish denseness of the subclass of all Archimedean copulas exibiting the pathological 
differentiability behavior sketched at the beginning of the introduction.  
Finally, Section \ref{section:category} provides the afore-mentioned category results for subfamilies 
of $\mathcal{C}_{ar}^d$. 

\section{Notation and Preliminaries}\label{section:notation}
\subsection{General notation and definitions}
Throughout this contribution we will write $\mathbb{I}=[0,1]$ and let bold symbols denote vectors.  
In what follows $\mathcal{C}^d$ with $d \geq 2$ denotes the family of all $d$-dimensional copulas, i.e., 
the family of distribution functions (restricted to $\mathbb{I}^d$) of random vectors 
$\mathbf{X}=(X_1,\ldots,X_d)$ fulfilling that each $X_i$ is uniformly distributed on $\mathbb{I}$.  
For a $d$-dimensional random vector $\mathbf{X}$ we will write $\mathbf{X}\sim C$ 
for some $C \in \mathcal{C}^d$ if $C$ is the joint distribution function of 
$\mathbf{X}$ restricted to $\mathbb{I}^{d}$. For every dimension $d \geq 2$, $M $ will denote the minimum copula
and $\Pi$ the product (independence) copula.  
Given $C \in \mathcal{C}^d$ we will let $\mu_C$ denote the corresponding $d$-stochastic measure, i.e., 
the probability measure defined by $\mu_C([\mathbf{0},\mathbf{x}]) := C(\mathbf{x})$ 
for all $\mathbf{x}=(x_1,\ldots,x_d) \in \mathbb{I}^d$ with $[\mathbf{0},\mathbf{x}] := [0,x_1] \times [0,x_2] \times\ldots\times [0,x_d]$ and extended to the Borel $\sigma$-field $\mathcal{B}(\mathbb{I}^d)$
on $\mathbb{I}^d$ in the standard measure-theoretic way. \\
For $\mathbf{X} \sim C$ the Kendall distribution function $F_K^d$ of $C$ is the distribution function 
of the (univariate) random variable $C(\mathbf{X})$, i.e., we have
$$F_K^d(t) := \mathbb{P}(C(\mathbf{X}) \leq t)$$
for every $t \in \mathbb{I}$.

Simplifying notation, for every pair $(i,j) \in \{1,...,d\}^2$ with $i<j$ and 
$\mathbf{x} \in \mathbb{I}^d$ we will write $\mathbf{x}_{i:j}:=(x_i,\ldots,x_j)$.  
Moreover, for $C \in \mathcal{C}^d$ and $1 \leq m \leq d$ the marginal copula $C^{1:m}$ of 
the first $m$ coordinates of $C$ is defined by
$C^{1:m}(x_1,x_2,\ldots, x_m) := C(x_1,x_2,\ldots, x_m, 1,\ldots, 1)$ for all $\mathbf{x} = (x_1,x_2,...,x_m) \in \mathbb{I}^m$.
The uniform metric $d_\infty$ on $\mathcal{C}^d$ is given by 
$$
d_\infty(A,B) := \sup_{\mathbf{x},\mathbf{y} \in \mathbb{I}^d}|A(\mathbf{x}) - B(\mathbf{y})|
$$
for all $A,B \in \mathcal{C}^d$. It is well known (see, e.g., \cite{dur_princ,nelsen2006}) 
that $(\mathcal{C}^d, d_\infty)$ is a compact metric space. 
For more background on copulas and $d$-stochastic measures we refer to \cite{dur_princ,nelsen2006}.

Given an arbitrary topological space $(S,\tau)$ the Borel $\sigma$-field on $S$ will be denoted 
by $\mathcal{B}(S)$, the family of all probability measures and of all (positive) finite measures on 
$\mathcal{B}(S)$ by $\mathcal{P}(S)$ and $\mathcal{M}(S)$, respectively.  
Focusing on Polish spaces $S$ the topology induced by weak convergence of elements in 
$\mathcal{P}(S)$ will be denoted by $\tau_w$ (see, e.g., \cite{Bill}).
For every $\nu \in \mathcal{M}(S)$ the support $\mathrm{supp}(\nu)$ of $\nu$ is 
the complement of the union of all open sets $U$ with the property that $\nu(U) = 0$. 
In the sequel the support of a copula 
$C \in \mathcal{C}^d$ is (by definition) the support of its corresponding $d$-stochastic measure $\mu_C$.

The Lebesgue measure on $\mathcal{B}(\mathbb{I}^d)$ or $\mathcal{B}(\mathbb{R}^d)$ 
will be denoted by $\lambda_d$, if the dimension is 
equal to $1$ we will drop the index and simply write $\lambda$ instead of $\lambda_1$. 
The Dirac measure at some point $x \in S$ is denoted by $\delta_x$.
Letting $(S,d)$ and $(S',d')$ denote two metric (or, more general, two topological) spaces, 
$T:S \rightarrow S'$ be a Borel-measurable transformation and $\nu \in \mathcal{P}(S)$, then the 
push-forward (measure) $\nu^T \in \mathcal{P}(S')$ of $\nu$ via $T$ is defined by
$\nu^T(F):=\nu(T^{-1}(F))$ for all $F \in \mathcal{B}(S')$.

In what follows conditional distributions and Markov kernels (regular conditional distributions) 
will play an important role. 
Given some $m \in \{1,\ldots,d-1\}$ we call a map 
$K: \mathbb{R}^m\times\mathcal{B}(\mathbb{R}^{d-m}) \rightarrow \mathbb{I}$ an $m$-\emph{Markov kernel} from 
$\mathbb{R}^m$ to $\mathbb{R}^{d-m}$, if the function $\mathbf{x}\mapsto K(\mathbf{x},E)$ is $\mathcal{B}(\mathbb{R}^{m})$-$\mathcal{B}(\mathbb{R}^{d-m})$-measurable for every fixed $E\in\mathcal{B}(\mathbb{R}^{d-m})$ and the map $E\mapsto K(\mathbf{x},E)$ is a probability measure on $\mathcal{B}(\mathbb{R}^{d-m})$ for every 
$\mathbf{x}\in\mathbb{R}^m$. 
If for every $\mathbf{x} \in \mathbb{I}^m$ the measure $K(\mathbf{x},\cdot)$ only fulfills 
$K(\mathbf{x},\mathbb{I}^{d-m}) \leq 1$ instead of $K(\mathbf{x},\mathbb{I}^{d-m}) = 1$ 
then we call $K$ an $m$-sub-Markov kernel.\\
Given a $(d-m)$-dimensional random vector $\mathbf{Y}$ and an $m$-dimensional random vector $\mathbf{X}$ 
on a joint probability space $(\Omega,\mathcal{A},\mathbb{P})$, a Markov kernel $K$ is called a 
regular conditional distribution of
$\mathbf{Y}$ given $\mathbf{X}$, if (and only if) for every set $E \in \mathcal{B}(\mathbb{R}^{d-m})$ the 
identity 
$$
K(\mathbf{X}(\omega), E) = \mathbb{E}(\mathbf{1}_E \circ \mathbf{Y} | \mathbf{X})(\omega)
$$
holds for $\mathbb{P}$-almost every $\omega \in \Omega$.
It is a well-known fact (see \cite{Kallenberg, Klenke}) that for each pair $(\mathbf{X}, \mathbf{Y})$ of random 
vectors such a regular conditional 
distribution $K$ of $\mathbf{Y}$ given $\mathbf{X}$ exists and that it is 
unique for $\mathbb{P}^{\mathbf{X}}$-almost every $\mathbf{x} \in \mathbb{R}^m$.
For $(\mathbf{X}, \mathbf{Y}) \sim C$ we will let 
$K_{C}:\mathbb{I}^m \times \mathcal{B}(\mathbb{I}^{d-m}) \to \mathbb{I}$ denote (a version of) the corresponding 
conditional distribution of $\mathbf{Y}$ given $\mathbf{X}$; $K_C$ will simply be referred to as 
(a version of) the $m$-\emph{Markov kernel} of the copula $C$. \\
For every $G \subseteq \mathbb{I}^d$ and $\mathbf{x} \in \mathbb{I}^m$ 
define the $\mathbf{x}$-section $G_\mathbf{x}$ of $G$ by $G_{\mathbf{x}}:=\{\mathbf{y}
\in \mathbb{I}^{d-m}: (\mathbf{x},\mathbf{y}) \in G\}\in\mathcal{B}(\mathbb{I}^{d-m})$. 
Applying \emph{disintegration} of $\mu_C$ into the marginal $\mu_{C^{1:m}}$ and the $m$-Markov kernel $K_C$ of $C$
(see \cite[Section 5]{Kallenberg} and \cite[Section 8]{Klenke}) the following identity 
holds for all $G \in \mathcal{B}(\mathbb{I}^d)$:
\begin{align}\label{eq:DI}
	\mu_C(G) = \int_{\mathbb{I}^m} K_{C}(\mathbf{x},G_{\mathbf{x}})
	\, \mathrm{d}\mu_{C^{1:m}}(\mathbf{x}).
\end{align}
For more background on conditional expectation, conditional 
distributions and Markov kernels we refer to \cite{Kallenberg, Klenke}. 

Throughout this article a measure $\nu \in \mathcal{M}(S)$ with $S=\mathbb{I}^d$ or $S=[0,\infty)$ 
is called singular (w.r.t. $\lambda_d$ or w.r.t. to $\lambda$, respectively)
if $\nu$ does not have any point-masses and if there exists some $G \in \mathcal{B}(S)$ such that 
$\lambda_d(G) = 0$  and $\nu(G) = \nu(S)$ hold. We will refer to such measures simply as `singular' (instead
of the alternative `singular without point-masses')
for the sake of simplicity, for being in accordance with singular distribution functions, and because
of the subsequent simple observation: 
For arbitrary $C \in \mathcal{C}^d$ the corresponding $d$-stochastic measure $\mu_C$ obviously has no point-masses, i.e., 
it's discrete component in the sense of Lebesgue decomposition is degenerated. 
Procee\-ding analogously to \cite{mult_arch} and decomposing the $(d-1)$-kernels $K_C^{}$, however, allows
to decompose $C$ into three not necessarily degenerated components: 
in fact, denoting the absolutely continuous, the discrete and the singular $(d-1)$-sub-kernels 
of $K_C$ by $K_C^{abs}$, $K_C^{dis}$, $K_C^{sing} \colon \mathbb{I}^{d-1} \times \mathcal{B}(\mathbb{I}) 
\rightarrow \mathbb{I}$, respectively, according to \cite{Lange} we have that
\begin{equation}\label{eq:lebesgue_decomp_markov}
K_C(\mathbf{x},F) = K_C^{abs}(\mathbf{x},F) + K_C^{dis}(\mathbf{x},F) + K_C^{sing}(\mathbf{x},F)
\end{equation}
holds for every $\mathbf{x} \in \mathbb{I}^{d-1}$ and every $F \in \mathcal{B}(\mathbb{I})$. Using
this decomposition and applying disintegration allows to define (to what we refer to as) 
the absolutely continuous, the discrete and the singular component of $\mu_C$ by
\begin{align}\label{eq:def_abs_dis_sing_copula}
\nonumber \mu_C^{abs}(E\times F) &:= \int_E K_C^{abs}(\mathbf{x},F) \mathrm{d}\mu_{C^{1:d-1}}(\mathbf{x}),\\
\nonumber\mu_C^{dis}(E\times F) &:= \int_E K_C^{dis}(\mathbf{x},F) \mathrm{d}\mu_{C^{1:d-1}}(\mathbf{x}),\\
\mu_C^{sing}(E\times F) &:= \int_E K_C^{sing}(\mathbf{x},F) \mathrm{d}\mu_{C^{1:d-1}}(\mathbf{x}),
\end{align}
for all $E \in \mathcal{B}(\mathbb{I}^{d-1})$ and $F \in  \mathcal{B}(\mathbb{I})$ and extend them 
to full $\mathcal{B}(\mathbb{I}^d)$ in the standard way. 
Throughout this article a copula $C \in \mathcal{C}^d$ with absolutely continuous $(d-1)$-marginal $C^{1:d-1}$ 
is called absolutely continuous, discrete or singular if $\mu_C^{abs}(\mathbb{I}^d) = 1$, 
$\mu_C^{dis}(\mathbb{I}^d) = 1$ or $\mu_C^{sing}(\mathbb{I}^d) = 1$ holds. 
%Furthermore, $\mu_C^{abs}$, $\mu_C^{dis}$ and $\mu_C^{sing}$ will be referred to as the absolutely continuous, 
%the discrete and singular component of $C$ (or $\mu_C$), respectively.

Considering a real function $u$, the left-hand and right-hand derivatives of $u$ (assuming their existence) 
will be denoted by $D^-u$ and $D^+u$, respectively; $u^{(m)}$ will denote the $m$-th derivative of $u$ (wherever it exists). 
A function $u\colon (a,b) \rightarrow \mathbb{R}$, whereby $(a,b) \subseteq \mathbb{R}$ denotes 
an open, not necessarily finite interval, is called $d$-monotone with $2 \leq d \in \mathbb{N}$, 
if (i) it is differentiable up to order $d-2$, if (ii) 
$$
(-1)^ku^{(k)}(x) \geq 0,
$$
holds for all $k \in \{0,...,d-2\}$ and $x \in (a,b)$, and if (iii) $(-1)^{d-2}u^{(d-2)}$ is non-increasing and convex on $(a,b)$. 
Finally, a real-valued function $u$ defined on $[0,\infty)$ is called $d$-monotone, if it is continuous on $[0,\infty)$ and $d$-monotone on $(0,\infty)$.
\subsection{Multivariate Archimedean copulas}\label{section:arch_cop}
In what follows we will consistently use the conventions 
$\inf \emptyset:=\infty, \frac{1}{\infty}:=0, \frac{1}{0}:=\infty$.  
We call a non-increasing and continuous function $\psi \colon [0,\infty) \rightarrow \mathbb{I}$ which fulfills $\psi(0) = 1$, $\lim_{z \rightarrow \infty}\psi(z) = 0 =: \psi(\infty)$ and which is 
strictly decreasing on $[0,\inf\{x\colon \psi(x) = 0\})$ an Archimedean generator; in the sequel 
we will simply refer to Archimedean generators as generators. 
The pseudo inverse $\varphi \colon \mathbb{I} \rightarrow [0,\infty]$ of a generator $\psi$ is defined by $\varphi(y) := \inf\{z \in [0,\infty] \colon \psi(z) = y\}$ for every $y \in \mathbb{I}$.
It is straightforward to see that $\varphi$ is convex and strictly decreasing on $(0,1]$, 
right-continuous at $0$ and fulfills $\varphi(1) = 0$ (see \cite{mult_arch}). 
A generator $\psi$ (or its pseudo-inverse $\varphi$) is called strict if $\varphi(0) = \infty$ (or equivalently if $\psi(z) > 0$ holds for all $z \in [0,\infty)$). 
A copula $C \in \mathcal{C}^d$ is called Archimedean if there exists an Archimedean generator $\psi$ such that
\begin{equation}\label{eq:def_arch_cop}
C(\mathbf{x}) := \psi\left(\sum_{i = 1}^d\varphi(x_i)\right)
\end{equation}
holds for all $\mathbf{x} \in \mathbb{I}^d$. In order to stress the correspondence between generator $\psi$ and copula 
$C$ in the sequel we sometimes write $C_\psi$ instead of $C$. 
According to \cite[Theorem 2.2]{neslehova}, $C(\mathbf{x}) := \psi\left(\sum_{i = 1}^d\varphi(x_i)\right)$
 is a $d$-dimensional Archimedean copula if, and only if $\psi$ is $d$-monotone on $[0,\infty)$, which, in turn is 
equivalent to the fact that $(-1)^{d-2}\psi^{(d-2)}$ exists on $(0,\infty)$, is non-negative, non-increasing 
and convex on $(0,\infty)$, see \cite[Proposition 2.3]{neslehova}. 
Convexity of $(-1)^{d-2}\psi^{(d-2)}$ implies that both $D^-\psi^{(d-2)}(z)$ and $D^+\psi^{(d-2)}(z)$ 
exist on $(0,\infty)$, that the one-sided derivatives coincide outside a countable set (see \cite[Theorem 3.7.4]{kannan1996} and \cite[Appendix C]{Pollard}) and that $D^-\psi^{(d-2)}(z) = D^+\psi^{(d-2)}(z)$ 
holds for every continuity point $z$ of $D^-\psi^{(d-2)}$. 
Furthermore (see \cite{mult_arch}), every $d$-monotone generator $\psi$ fulfills 
$\psi^{(m)}(\infty):=\lim_{z \rightarrow \infty}\psi^{(m)}(z) = 0$ for every $m \in \{0,...,d-2\}$ as well as 
$D^-\psi^{(d-2)}(\infty):=\lim_{z \rightarrow \infty}D^-\psi^{(d-2)}(z) = 0$.\\
In order to have a one-to-one correspondence between generators and Archimedean copulas in what follows 
we always assume the generator $\psi$ to be normalized in the sense that $\psi(1) = \frac{1}{2}$ (or equivalently $\varphi(\frac{1}{2}) = 1$) holds and denote the family of all normalized $d$-monotone generators by $\mathbf{\Psi}_d$. 
The class of all $d$-dimensional Archimedean copulas will be denoted by $\mathcal{C}_{ar}^d$.  
Moreover an Archimedean copula $C_\psi \in \mathcal{C}_{ar}^d$ will be called strict (non-strict) 
if its corresponding generator $\psi$ is strict (non-strict). Throughout this contribution $\mathcal{C}_{ar,s}^d$ and $\mathcal{C}_{ar,n}^d$ will denote
the subclasses of all strict and non-strict $d$-dimensional Archimedean copulas, respectively. 
According to \cite{neslehova} $C \in \mathcal{C}_{ar}^d$ is absolutely continuous if, and only if $\psi^{(d-1)}$ exists and is absolutely continuous on $(0,\infty)$. Building upon this fact, in the absolutely 
continuous case (a version of) the density $c$ of C is given by
\begin{equation}\label{eq:density}
c(\mathbf{x}) = \mathbf{1}_{(0,1)^d}(\mathbf{x})\prod_{i = 1}^d \varphi'(x_i) \cdot D^-\psi^{(d-1)}\left(\sum_{i = 1}^d \varphi(x_i)\right)
\end{equation}
for all $\mathbf{x} \in \mathbb{I}^d$. 
As established in \cite[Proposition 4.1]{neslehova}, every at most $(d-1)$-dimensional 
marginal of a $d$-dimensional Archimedean copula is absolutely continuous. 
This fact allows us to work with the definition of the absolutely continuous, the discrete and the 
singular component of an Archimedean copula $C \in \mathcal{C}_{ar}^d$ according to 
equation \eqref{eq:def_abs_dis_sing_copula}; the families of all absolutely continuous, of all discrete and 
of all singular Archimedean copulas will be denoted by $\mathcal{C}_{ar,abs}^d$, $\mathcal{C}_{ar,dis}^d$ and $\mathcal{C}_{ar,sing}^d$, respectively. \\
Defining the family $\mathcal{P}_{nor}^d$ of all $d$-normalized probability measures by
$$
\mathcal{P}_{nor}^d := \left\{ \gamma \in \mathcal{P}([0,\infty)):\int_{\mathbb{I}}(1-t)^{d-1} \mathrm{d}\gamma(t) = \frac{1}{2} \right\},
$$
it is straightforward to verify that $\mathcal{P}_{nor}^d$ is weakly closed in $\mathcal{P}([0,\infty))$.
Considering that $\mathcal{P}([0,\infty))$ is a complete metric space (see \cite{Parthasarathy}), 
$\mathcal{P}_{nor}^d$ is (as closed subset of a complete metric space) complete as well.
Motivated by the results in \cite{neslehova}, characterizing Archimedean copulas via so-called Williamson measures
we define the family of all normalized $d$-Williamson measures $\mathcal{P}_{\mathcal{W}_d}$ by
\begin{equation}
	\mathcal{P}_{\mathcal{W}_d}:=\left\{ \gamma \in \mathcal{P}([0,\infty)): \gamma(\{0\}) = 0
	\textrm{ and } \int_{\mathbb{I}}(1-t)^{d-1} \mathrm{d}\gamma(t) = \frac{1}{2} \right\}.
\end{equation}
If the dimension $d$ is clear from the context we simply speak of Williamson measures. 
Throughout this article we use the convention that $\gamma(\{\infty\}) := 0$.
Following \cite{mult_arch,neslehova} there is a one-to-one correspondence between the family of normalized $d$-Williamson-measures $\mathcal{P_{W_d}}$ and the family $\mathbf{\Psi}_d$ of all normalized $d$-monotone 
Archimedean generators. In fact, the mapping $\mathcal{W}_d$ (usually referred to as the Williamson $d$-transform),
defined by
\begin{equation}\label{eq:williamson_transform}
\psi(z):=(\mathcal{W}_d\gamma)(z) := \int_{[0,\infty)}(1-tz)_+^{d-1}\mathrm{d}\gamma(t), \quad z \in [0,\infty),
\end{equation}
maps $(\mathcal{P}_{\mathcal{W}_d},\tau_w)$ to $(\Psi_d,\Vert \cdot \Vert_\infty)$. 
Building upon the afore-mentioned one-to-one correspondence between normalized $d$-monotone Archimedean generators $\psi \in \mathbf{\Psi}_d$ 
and $d$-dimensional Archimedean copulas $C \in \mathcal{C}_{ar}^d$ the map $\xi_d$ implicitly 
defined by equation \eqref{eq:def_arch_cop} maps $(\mathcal{W}_d, \tau_w)$ to $(\Psi_d,\Vert \cdot \Vert_\infty)$. 
It is straightforward to show that both maps $\mathcal{W}_d$ and $\xi_d$ are homeomorphisms.
\begin{Lemma}\label{lem:hom_archimedean}
    The maps $\mathcal{W}_d$, $\xi_d$ and $\xi_d \circ \mathcal{W}_d$ are homeomorphisms.
\end{Lemma}
\begin{proof}
According to \cite[Theorem 5.1.]{mult_arch} the map $\mathcal{W}_d$ is a bijection, hence 
applying \cite[Theorem 5.9.]{mult_arch} yields that both $\mathcal{W}_d$ and its inverse are continuous. 
Working with normalized gene\-rators and again following \cite{mult_arch} implies that the map $\xi_d$ is bijective.
Using \cite[Theorem 4.1]{mult_arch} yields continuity of $\xi_d$ and its inverse $\xi_d^{-1}$. 
Since compositions of homeomorphisms are homeomorphisms, this completes the proof.
\end{proof}
It is straightforward to derive from equation (\ref{eq:williamson_transform}) that for every $r \in (0,\infty)$
we have $\gamma([0,r))=0$ if, and only if $\psi(\frac{1}{r})=0$. Moreover, 
upon \cite[Lemma 5.5]{mult_arch} stating that $C \in \mathcal{C}_{ar}^d$ is strict if, and only if
the support of the corresponding Williamson measure $\gamma \in \mathcal{P}_{\mathcal{W}_d}$ contains the point $0$, 
we may define the family of all Williamson measures associated with strict $d$-dimensional Archimedean copulas by
    $$
    \mathcal{P}_{\mathcal{W}_d}^s := \{\gamma \in \mathcal{P}_{\mathcal{W}_d} \colon \gamma((0,r)) > 0 \text{ for every } r>0\}
    $$
    and the family of all Williamson measures associated with non-strict Archimedean copulas by
    $$
    \mathcal{P}_{\mathcal{W}_d}^n := \{\gamma \in \mathcal{P}_{\mathcal{W}_d} \colon \exists r > 0 \text{ with } \gamma((0,r)) = 0\}.
    $$
Finally, $\mathcal{P}_{\mathcal{W}_d}^{fs}$ will denote the set of all normalized $d$-Williamson measures 
having full support $[0,\infty)$, the families of all (purely) absolutely continuous, discrete and 
singular normalized $d$-Williamson measures will be denoted by $\mathcal{P}_{\mathcal{W}_d}^{abs}$, $\mathcal{P}_{\mathcal{W}_d}^{dis}$ and $\mathcal{P}_{\mathcal{W}_d}^{sing}$, respectively.

Throughout this article we will frequently work with the level sets of Archimedean copulas $C$, i.e., the sets 
of all points $\mathbf{x} \in \mathbb{I}^d$ with $C(\mathbf{x})=t$ for some predefined $t \in \mathbb{I}$. 
For fixed $t \in (0,1]$ the $t$-level set $L_t^{1:d}:=L_t$ of $C \in \mathcal{C}_{ar}^d$ is given by
\begin{align}
\nonumber L_t &:= \{(\mathbf{x},y) \in \mathbb{I}^{d-1} \times \mathbb{I} \colon C(\mathbf{x},y) = t\} \\& =
\left\{(\mathbf{x},y) \in \mathbb{I}^{d-1} \times \mathbb{I} \colon \sum_{i=1}^{d-1}\varphi(x_i) + \varphi(y) = \varphi(t)\right\}
\end{align}
and for $t = 0$ by
\begin{align}
\nonumber L_0 &:= \{(\mathbf{x},y) \in \mathbb{I}^{d-1} \times \mathbb{I} \colon C(\mathbf{x},y) = 0\} \\& =
\left\{(\mathbf{x},y) \in \mathbb{I}^{d-1} \times \mathbb{I} \colon \sum_{i=1}^{d-1}\varphi(x_i) + \varphi(y) \geq \varphi(0)\right\}.
\end{align}
The $t$-level sets of the marginal copula $C^{1:d-1}$ are defined analogously and denoted by $L_t^{1:d-1}$ for 
every $t \in \mathbb{I}$. Moreover, the function $f^0\colon \mathbb{I}^{d-1} \rightarrow \mathbb{I}$, whose graph coincides with the upper boundary of $L_0$, is defined by
\begin{align}\label{eq:f0}
f^0(\mathbf{x}) &:= \begin{cases}
    1,& \text{ if } \mathbf{x} \in L_0^{1:d-1},\\
    \psi\left(\varphi(0) - \sum_{i = 1}^{d-1}\varphi(x_i)\right),& \text{ if } \mathbf{x} \not\in L_0^{1:d-1},
\end{cases}
\end{align}
whereby we set $\psi(u) = 1$ for all $u < 0$. Notice that equation (\ref{eq:f0}) holds both, in the strict and the 
non-strict setting: In the strict case we have $\varphi(0)=\infty$, so for $\mathbf{x} \not\in L_0^{1:d-1}$
we have $\sum_{i = 1}^{d-1}\varphi(x_i) \in [0,\infty)$, implying $\varphi(0) - \sum_{i = 1}^{d-1}\varphi(x_i)=\infty$ and $f^0(\mathbf{x})=\psi(\infty)=0$; hence in the strict case we 
have $f^0(\mathbf{x}) \in \{0,1\}$ for every $\mathbf{x} \in \mathbb{I}^{d-1}$. Contrary to that, in 
the non-strict case $f^0$ also attains values in $(0,1)$. \\
For $t \in (0,1]$ we will also use the upper $t$-cut $[C]_t$ of $C$, defined by 
$$[C]_t := \{\mathbf{x} \in \mathbb{I}^{d} \colon C(\mathbf{x}) \geq t\}.$$
The upper $t$-cuts of marginal copulas are defined analogously.  
The so-called $t$-level function $f^t \colon [C^{1:d-1}]_t \rightarrow \mathbb{I}$ is given by
\begin{equation}
f^t(\mathbf{x}) :=
    \psi\left(\varphi(t) - \sum_{i=1}^{d-1}\varphi(x_i)\right)
\end{equation}
for every $\mathbf{x} \in [C^{1:d-1}]_t$. It is straightforward to verify that for $t \in (0,1]$ the graph of $f^t$ coincides with $L_t$. 
Moreover, as proved in the next section, the discrete component of 
an Archimedean copula (if any) is always concentrated on the graphs of some $f^t$ with $t \in [0,1]$. 
According to \cite{mult_arch,neslehova} the mass of the $t$-level sets can be explicitly calculated - in fact, 
for every $t \in (0,1)$ 
\begin{equation}\label{eq:lvl_set_mass_derivative_t}
    \mu_C(L_t) =  \frac{(-\varphi(t))^{d-1}}{(d-1)!}(D^-\psi^{(d-2)}(\varphi(t)) - D^-\psi^{(d-2)}(\varphi(t)+))
\end{equation}
holds and for $t=0$ we have  
\begin{equation}\label{eq:lvl_set_mass_derivative_0}
    \mu_C(L_0) =  \frac{(-\varphi(0))^{d-1}}{(d-1)!}D^-\psi^{(d-2)}(\varphi(0))
\end{equation}
in the non-strict case as well as $\mu_C(L_0) = 0$ in the strict case. For $t=1$ we obviously have 
$L_1=\{(1,\ldots,1)\}$, implying $\mu_C(L_1)=0$.  \\
Given $0<s<t\leq 1$, define the set 
$$
L_{[s,t]} := \left\{(\mathbf{x},y) \in \mathbb{I}^{d-1} \times \mathbb{I} \colon C(\mathbf{x},y) \in [s,t] \right\}.
$$
Extending the definition of $f^s$ to full $\mathbb{I}^{d-1}$ by setting $f^s(\mathbf{x}) = 1$ for 
$\mathbf{x}\in  \mathbb{I}^{d-1} \setminus [C^{1:d-1}]_s$ 
obviously the set $L_{[s,t]}$ may also be expressed as 
$$
L_{[s,t]} = \left\{(\mathbf{x},y) \in \mathbb{I}^{d-1} \times \mathbb{I}: f^s(\mathbf{x}) \leq y \leq f^t(\mathbf{x}) \right\}.
$$

Considering a $d$-dimensional Archimedean copula $C$ with generator $\psi$, then according to \cite{mult_arch,neslehova} the Kendall distribution function $F_K^d$ of $C$ can explicitly be calculated 
and is given by
$$
F_K^d(t) = D^-\psi^{(d-2)}(\varphi(t))\cdot\frac{(-1)^{d-1}}{(d-1)!}\varphi(t)^{d-1} + \sum_{k = 0}^{d-2}\psi^{(k)}(\varphi(t))\frac{(-1)^k}{k!}\varphi(t)^k
$$
for every $t \in (0,1]$ and by
$$
F_K^d(0) = \mu_C(L_0) =  \frac{(-\varphi(0))^{d-1}}{(d-1)!}D^-\psi^{(d-2)}(\varphi(0))
$$
for $t = 0$. Throughout this contribution we denote the probability measure corresponding to the 
Kendall distribution function $F_K^d$ by $\kappa_{F_K^d}$.
As proved in \cite[Theorem 5.6]{mult_arch}, the Kendall distribution function $F_K^d$
and the mass of the $t$-level sets $L_t$ can be represented directly via the Williamson measure.  
The following result is a compressed version of \cite[Theorem 5.6]{mult_arch} using the conventions mentioned 
at the beginning of Section 2.2. 
\begin{theorem}\label{thm:lvl_sets_kendall_williamson}
    Let $C$ be a $d$-dimensional Archimedean copula with generator $\psi$ and Williamson measure $\gamma$. 
    Then (in the strict and the non-strict case) we have that
     \begin{equation}\label{eq:t-level-set}
        \mu_C(L_t) = \gamma\left(\left\{\frac{1}{\varphi(t)}\right\}\right)
    \end{equation}
holds for every $t\in \mathbb{I}$. 
Moreover, the Kendall distribution function $F_K^d$ of $C$ fulfills
\begin{equation}\label{eq:kendall_williamson}
F_K^d(t) = \gamma\left(\left[0,\frac{1}{\varphi(t)}\right]\right)
\end{equation}
for every $t \in \mathbb{I}$.
\end{theorem}

We conclude this section with an explicit expression for the $(d-1)$-Markov kernel of Archimedean copulas as
derived in \cite[Theorem 3.1]{mult_arch}; this expression will be key in the next sections.
Suppose that $C \in \mathcal{C}_{ar}^d$ and let $\psi$ and $\varphi$ be its associated generator and pseudo-inverse, respectively. Then (a version of) the Markov-kernel $K_C$ of $C$ is given by
    \begin{align}\label{eq:markov_kernel_arch}
    K_C(\mathbf{x},[0,y]) := \begin{cases}
        1,& M(\mathbf{x}) = 1 \text{ or } \mathbf{x} \in L_0^{1:d-1}\\
        0,& M(\mathbf{x}) < 1, \mathbf{x} \notin L_0^{1:d-1},y < f^0(\mathbf{x})\\
        \frac{D^-\psi^{(d-2)}\left(\sum_{i=1}^{d-1}\varphi(x_i) + \varphi(y)\right)}{D^-\psi^{(d-2)}\left(\sum_{i=1}^{d-1}\varphi(x_i)\right)},& M(\mathbf{x}) < 1, \mathbf{x}\notin L_0^{1:d-1},y \geq f^0(\mathbf{x}),
    \end{cases}
    \end{align}
    for every $\mathbf{x} \in \mathbb{I}^{d-1}$ and every $y \in \mathbb{I}$, see \cite[Theorem 3.1]{mult_arch}.
Notice that (again see \cite{mult_arch}) in the first line of equation (\ref{eq:markov_kernel_arch}) 
we could replace the constant $1$ by $F(y)$ for any univariate distribution function $F$ since 
$L_0^{1:d-1}$ fulfills $\mu_{C^{1:d-1}}(L_0^{1:d-1})=0$ and the $(d-1)$-kernel $K_C$ is only unique outside a set of 
$\mu_{C^{1:d-1}}$-measure $0$.
     
\section{Regularity and mass distributions of multivariate Archimedean copulas}\label{section:regularity}
Looking at the Williamson $d$-transform (\ref{eq:williamson_transform}) not surprisingly 
$D^{-}\psi^{(d-2)}$ can also be directly expressed in terms of the corresponding Williamson measure. 
In fact, following \cite[Lemma 5.4]{mult_arch} for every $d$-dimensional Archimedean copula $C$ with generator
$\psi$ and (normalized) Williamson measure $\gamma$ the subsequent lemma holds: 
\begin{Lemma}\label{lem:function_G}
    Let $C \in \mathcal{C}_{ar}^d$ be an Archimedean copula and $\psi$ and $\gamma \in \mathcal{P}_{\mathcal{W}_d}$ be its corresponding generator and Williamson measure, respectively. Then
    \begin{equation}\label{eq:function_G}
    0 \geq G(z) := (-1)^{d-2}D^{-}\psi^{(d-2)}(z) = -(d-1)!\int_{(0,\frac{1}{z}]} t^{d-1} \mathrm{d}\gamma(t)
    \end{equation}
    holds for every $z > 0$.
\end{Lemma}
Building upon the previous lemma and fixing $\gamma \in \mathcal{P}_{\mathcal{W}_d}$, 
for the case $M(\mathbf{x})<1, \,\mathbf{x} \not\in L_0^{1:d-1}$ and $f^0(\mathbf{x}) \leq y$ 
the Markov kernel $K_C$ of $C \in \mathcal{C}^d_{ar}$ according to equation \eqref{eq:markov_kernel_arch} can 
be expressed as
\begin{equation}\label{eq:kernel_function_G}
K_C(\mathbf{x},[0,y]) = \frac{G(\sum_{i = 1}^{d-1}\varphi(x_i) + \varphi(y))}{G(\sum_{i = 1}^{d-1}\varphi(x_i))} = \frac{\int_{I_y}t^{d-1}\mathrm{d}\gamma(t)}{\int_{I_1}t^{d-1}\mathrm{d}\gamma(t)},
\end{equation}
with the notation $I_y := \left(0,\frac{1}{\sum_{i=1}^{d-1}\varphi(x_i) + \varphi(y)}\right]$ for every $y \in \mathbb{I}$. 
The last expression in equation (\ref{eq:kernel_function_G}) is convenient insofar that 
(as proved in this section) it insinuates that the regularity of the measure $\gamma$ directly 
propagates to the Markov kernel $K_C$.
Consi\-dering the Lebesgue decomposition $\gamma = \gamma^{abs} + \gamma^{dis} + \gamma^{sing}$ of $\gamma \in \mathcal{P}_{\mathcal{W}_d}$, equation \eqref{eq:kernel_function_G} can be written as
\begin{align}\label{eq:kernel_function_G_abs_dis_sing}
\nonumber K_C(\mathbf{x},[0,y]) &= \frac{\int_{I_y}t^{d-1}\mathrm{d}\gamma(t)}{\int_{I_1}t^{d-1}\mathrm{d}\gamma(t)} \\&=
\underbrace{\frac{\int_{I_y}t^{d-1}\mathrm{d}\gamma^{abs}(t)}{\int_{I_1}t^{d-1}\mathrm{d}\gamma(t)}}_{=:H^{abs}_\mathbf{x}(y)} + 
\underbrace{\frac{\int_{I_y}t^{d-1}\mathrm{d}\gamma^{dis}(t)}{\int_{I_1}t^{d-1}\mathrm{d}\gamma(t)}}_{=:H^{dis}_\mathbf{x}(y)} + 
\underbrace{\frac{\int_{I_y}t^{d-1}\mathrm{d}\gamma^{sing}(t)}{\int_{I_1}t^{d-1}\mathrm{d}\gamma(t)}.}_{=:H^{sing}_\mathbf{x}(y)}
\end{align}
Notice that the definitions for $H^{abs}_\mathbf{x}, H^{dis}_\mathbf{x}, H^{sing}_\mathbf{x}$ in 
equation (\ref{eq:kernel_function_G_abs_dis_sing}) are merely definitions, it is a priori not clear
that the chosen notation is meaningful, i.e., that the function $H^{abs}_\mathbf{x}$ (if being non-degenerated) 
is absolutely continuous, that 
$H^{dis}_\mathbf{x}$ (if being non-degenerated) is discrete and that $H^{sing}_\mathbf{x}$ (if being non-degenerated) 
is singular. These points, however, will be established in the proofs of Lemma \ref{lem:absolutely_cont}, 
Lemma \ref{lem:sing} and in Remark \ref{rem:discrete}. \\

It has already been established (see \cite[Theorem 5.12]{mult_arch}) that absolute continuity, discreteness 
and singularity of the Williamson-measure $\gamma \in \mathcal{P}_{\mathcal{W}_d}$ propagates to the corresponding Archimedean copula $C \in \mathcal{C}_{ar}^d$. 
The goal of this section is to extend this result to the following equivalence 
(see Theorem \ref{thm:main_regularity_result}): 
the absolutely continuous/discrete/singular component of the Williamson measure $\gamma$ is non-degenerated if, 
and only if the absolutely continuous/discrete/singular component of the corresponding Archimedean copula $C$ is non-degenerated. Having this stronger result provides a simple proof of the reverse implication in 
\cite[Theorem 5.12]{mult_arch}. \\

We start with some preliminary observations which will be key for proving the afore-mentioned equivalence.   
The following Theorem is a consequence of Theorem \ref{thm:lvl_sets_kendall_williamson}.
\begin{theorem}\label{thm:mass_ex_lvl_set}
    Let $C$ be a $d$-dimensional Archimedean copula and $\gamma \in \mathcal{P}_{\mathcal{W}_d}$ its corresponding Williamson measure. Moreover, let $s_1,s_2 \in \mathbb{I}$, with $s_1 \leq s_2$. Then the following assertion holds:
    \begin{equation}\label{eq:mass_extended_level_set}
    \mu_C(L_{[s_1,s_2]}) = \gamma\left(\left[\frac{1}{\varphi(s_1)},\frac{1}{\varphi(s_2)}\right]\right).
    \end{equation}
\end{theorem}
\begin{proof}
    Setting $M_t := \{(\mathbf{x},y) \in \mathbb{I}^d\colon C(\mathbf{x},y) \leq t\}$ for every 
    $t \in \mathbb{I}$ we obviously have $F_K^d(t) = \mu_C(M_t)$ for every $t \in \mathbb{I}$. 
    For $s_1 \leq s_2$ considering 
    $$
    L_{[s_1,s_2]} = M_{s_2} \cap [C]_{s_1} = M_{s_2} \setminus [C]_{s_1}^c, \quad [C]_{s_1}^c \subseteq M_{s_2},
    $$ 
    using Theorem 
    \ref{thm:lvl_sets_kendall_williamson} directly yields
    \begin{align*}
    \mu_C(L_{[s_1,s_2]}) &= \mu_C(M_{s_2}\setminus [C]_{s_1}^c) \\&=
    \mu_C(M_{s_2}) - \mu_C([C]_{s_1}^c) \\&=
    F_K^d(s_2) - F_K^d(s_1-) \\&=
    \gamma\left(\left[\frac{1}{\varphi(s_1)}, \frac{1}{\varphi(s_2)}\right]\right).
     \end{align*}
   \end{proof}
The next corollary states a close interrelation between the support of the Williamson measure 
$\gamma \in \mathcal{P}_{\mathcal{W}_d}$ and the support of the corresponding Archimedean copula 
$C \in \mathcal{C}_{ar}^d$. 
\begin{Cor}\label{cor:support_archimedean}
    Suppose that $C \in \mathcal{C}_{ar}^d$, let $\gamma \in \mathcal{P}_{\mathcal{W}_d}$ be its associated Williamson measure, $\psi$ its generator, and $\varphi$ the pseudo-inverse of $\psi$. 
    Then for $[a,b] \subseteq [\frac{1}{\varphi(0)},\infty)$ with $a<b$ the  
    following two conditions are equivalent:
    \begin{enumerate}
     \item $\gamma([a,b]) = 0$
     \item $\mu_C(L_{[\psi(\frac{1}{a}),\psi(\frac{1}{b})]}) = 0$
    \end{enumerate}
   \end{Cor}
\begin{proof}
Immediate consequence of Theorem \ref{thm:mass_ex_lvl_set}.
\end{proof}
\begin{Rem}
\emph{The condition $[a,b] \subseteq [\frac{1}{\varphi(0)},\infty)$ can not be avoided. 
In fact, in the non-strict setting for the situation $b < \frac{1}{\varphi(0)}$ 
it follows that $\gamma([0,b])=0$, implying $\gamma([a,b])=0$. Nevertheless in this case we have 
$\psi(\frac{1}{a})=\psi(\frac{1}{b})=0$ but $\mu_C(L_{[0,0]})>0$ may hold.}
\end{Rem}
The subsequent example illustrates Corollary \ref{cor:support_archimedean}.
\begin{Ex}\label{ex:supp_arch_luecken}
\normalfont
    Consider the distribution function
     \begin{align*}
        F_\gamma(z) := \begin{cases}
        0, &\text{ if } z \in [0,\frac{1}{4}),\\
        \frac{2}{3}, &\text{ if } z \in [\frac{1}{4},1),\\
        \frac{1}{9}z+\frac{5}{9}, &\text{ if } z \in [1,2),\\
        \frac{7}{9}, &\text{ if } z \in [2,3),\\
        \frac{1}{9}z+\frac{4}{9}, &\text{ if } z \in [3,4),\\
        1, &\text{ if } z \geq 4.
        \end{cases}
    \end{align*}
    It is straightforward to verify that $F_\gamma$ is the distribution function of a unique 
    Williamson measure $\gamma \in \mathcal{P}_{\mathcal{W}_2}$, whose  
    induced non-strict generator $\psi$ is given by
     $$
     \psi(z) := \begin{cases}
        1 - \frac{7z}{6},  &\text{ if } z \in [0,\frac{1}{4}],\\
        \frac{3z^2+8z+1}{18z}, &\text{ if } z \in (\frac{1}{4},\frac{1}{3}],\\
        \frac{1}{9}(7-3z), &\text{ if } z \in (\frac{1}{3},\frac{1}{2}],\\
        -\frac{2z^2-10z-1}{18z}, &\text{ if } z \in (\frac{1}{2},1],\\
        \frac{2}{3}(1-\frac{z}{4}), &\text{ if } z \in (1,4],\\
        0, &\text{ if } z > 4.
        \end{cases}
    $$
The distribution function $F_\gamma$ and the generator $\psi$ are depicted in Figure \ref{fig:meas_gen_luecken}. Considering the support of the corresponding Archimedean copula $C_\psi$, we obtain that
    \begin{align*}
    \mathrm{supp}(\mu_{C_\psi}) &=
    \Gamma(f^0) \cup L_{[\frac{1}{2},\frac{11}{18}]} \cup L_{[\frac{2}{3},\frac{17}{24}]}.
    \end{align*}
    A sample of $C_\psi$ is depicted in Figure \ref{fig:dis_abs_cont_arch_luecken}.
\end{Ex}
\begin{figure}[!ht]
	\centering
\includegraphics[width=1\textwidth]{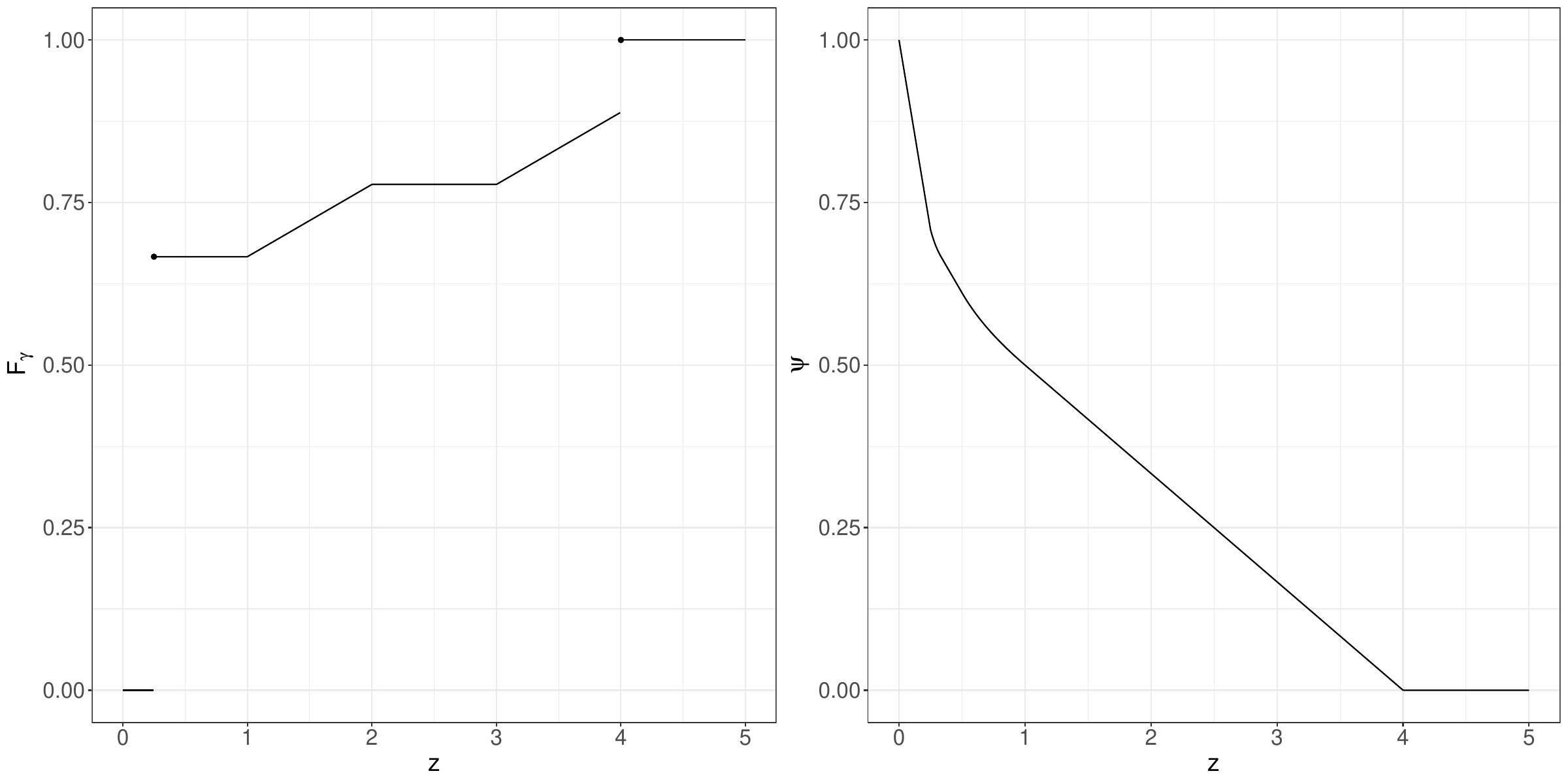}
	\caption{The distribution function $F_\gamma$ (left panel) and the associated generator $\psi$
	 (right panel) considered in Example \ref{ex:supp_arch_luecken}.}
\label{fig:meas_gen_luecken}
\end{figure}
\begin{figure}[!ht]
	\centering
\includegraphics[width=1\textwidth]{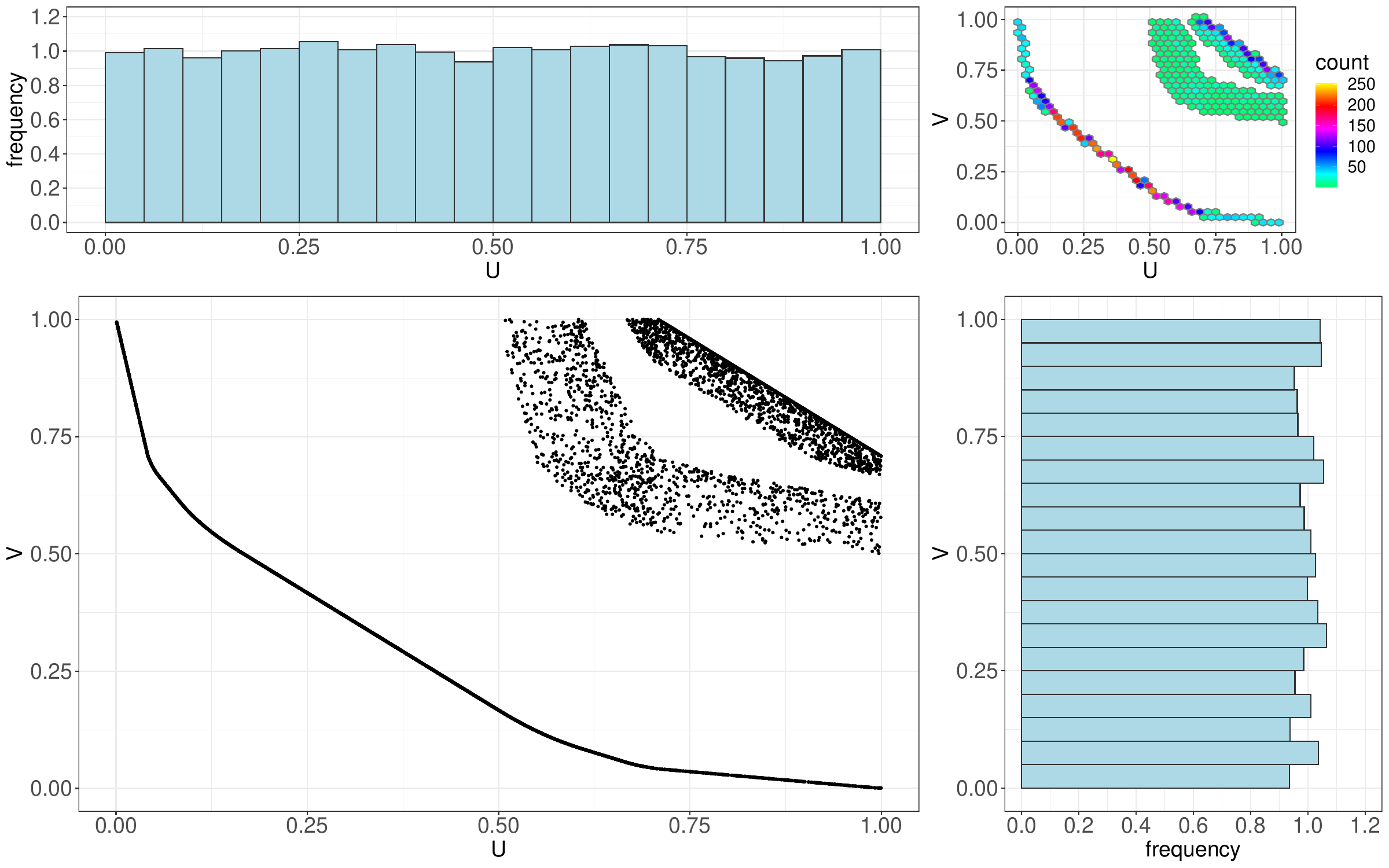}
	\caption{Sample of size 10000 of the Archimedean copula $C_\psi$ with $\psi$ being 
	the generator from Example \ref{ex:supp_arch_luecken}, its histogram and the two marginal histograms;  
     sample generated via conditional inverse sampling.}
\label{fig:dis_abs_cont_arch_luecken}
\end{figure}
According to \cite[Lemma 5.5]{mult_arch} strictness of an Archimedean copula $C$ can be characterized in terms 
of the corresponding Williamson measure. 
Corollary \ref{cor:support_archimedean} allows to extend this result and show that an Archimedean copula 
$C \in \mathcal{C}_{ar}^d$ has full support if, and only if its corresponding Williamson measure 
$\gamma$ has full support.
\begin{theorem}\label{thm:arch_full_support}
Let $C \in \mathcal{C}_{ar}^d$ and $\gamma \in \mathcal{P}_{\mathcal{W}_d}$ be its corresponding Williamson measure. Then $\gamma$ has full support $[0,\infty)$ if, and only if $C$ has full support $\mathbb{I}^d$.
\end{theorem}
\begin{proof}
    (i) Suppose that $\gamma \in \mathcal{P}_{\mathcal{W}_d}$ has full support. 
    Then its corresponding distribution function $F_\gamma$ is strictly increasing, so according 
    to \cite[Lemma 5.5]{mult_arch} the corresponding Archimedean copula $C \in \mathcal{C}_{ar}^d$ is strict. 
    Fix $\mathbf{x} \notin L_0^{1:d-1}$ with $M(\mathbf{x}) < 1$ and suppose that $y_1,y_2 \in \mathbb{I}$ 
    fulfill $f^0(\mathbf{x})=0 \leq y_1 < y_2 \leq 1$. 
    Using the fact that $\gamma$ has full support we have $\gamma(I_{y_2} \setminus I_{y_1}) > 0$, which implies 
    \begin{align*}
    \int_{I_{y_1}}t^{d-1}\mathrm{d}\gamma(t) < \int_{I_{y_2}}t^{d-1}\mathrm{d}\gamma(t).
    \end{align*}
    Having this, using equation (\ref{eq:kernel_function_G}) $K_C(\mathbf{x},(y_1,y_2]) > 0$ follows
    and we have shown that the conditional distribution function $y \mapsto K_C(\mathbf{x},[0,y])$ 
    is strictly increasing. Considering $\mu_{C^{1:d-1}}((L_0^{1:d-1})^c \cap M^{-1}([0,1))) = 1$ 
    it follows that for $\mu_{C^{1:d-1}}$-almost every $\mathbf{x} \in \mathbb{I}^{d-1}$
    the conditional distribution function $y \mapsto K_C(\mathbf{x},[0,y])$ is strictly increasing.  
    Letting $R=(\underline{x}_1,\overline{x}_1) \times (\underline{x}_2,\overline{x}_2) \times \cdots 
    \times (\underline{x}_{d-1},\overline{x}_{d-1}) \times (\underline{y},\overline{y}) \subset \mathbb{I}^d$ 
    denote an arbitrary open rectangle with $\lambda_d(R)>0$, writing 
    $R^{1:d-1}:=(\underline{x}_1,\overline{x}_1) \times (\underline{x}_2,\overline{x}_2) \times \cdots 
    \times (\underline{x}_{d-1},\overline{x}_{d-1}) $ and applying disintegration yields 
    \begin{align*}
    \mu_C(R) = \int_{R^{1:d-1}} \underbrace{K_C(\mathbf{x},(\underline{y},\overline{y}))}_{>0} 
    \mathrm{d}\mu_{C^{1:d-1}}(\mathbf{x})>0.
\end{align*}     
In other words: Every open rectangle with positive volume has positive mass. 
This shows $\mathrm{supp}(\mu_C) = \mathbb{I}^d$ and completes the proof of the first implication. \\
Considering the reverse direction, suppose that $\gamma$ would not have full support. Then we could find 
some $a,b \in [0,\infty)$ with $a < b$ such that $\gamma([a,b]) = 0$. If 
$[a,b] \subseteq [\frac{1}{\varphi(0)},\infty)$,  Corollary \ref{cor:support_archimedean} directly yields 
a contradiction. If $a < \frac{1}{\varphi(0)}$ then $\varphi$ has to be non-strict, implying 
that $L_0$ has non-empty interior, so $\mathrm{supp}(\mu_C) \neq \mathbb{I}^d$.
\end{proof}
As next step we prove that the discrete component (if any) of an Archimedean copula $C$ is always concentrated on the graph of the $t$-level set functions $f^t$ for some $t \in [0,1)$. 
\begin{Lemma}\label{lem:equivalence_point_mass}
    Let $C \in \mathcal{C}_{ar}^d$, $\psi$ be its generator, $\gamma \in \mathcal{P}_{\mathcal{W}_d}$ its corresponding Williamson measure and consider $t_0 \in (0,1)$ in case that $\psi$ is strict and $t_0 \in [0,1)$ in case that $\psi$ is non-strict. Then the following assertions are equivalent:
    \begin{itemize}
        \item[(i)]  $\gamma\left(\left\{\frac{1}{\varphi(t_0)}\right\}\right)>0$,
        \item[(ii)] $\varphi(t_0)$ is a point of discontinuity of $D^{-}\psi^{(d-2)}$,
        \item[(iii)] There exists some set $\Lambda \in \mathcal{B}(\mathbb{I}^{d-1})$ with $\mu_{C^{1:d-1}}(\Lambda) > 0$ such that for all $\mathbf{x} \in \Lambda$ $K_C(\mathbf{x},\{f^{t_0}(\mathbf{x})\}) > 0$ holds,
        \item[(iv)] $\mu_C(L_{t_0}) > 0$.
    \end{itemize}
\end{Lemma}
\begin{proof}
The equivalence of the first, second and fourth assertion is an immediate consequence of Theorem \ref{thm:lvl_sets_kendall_williamson} and equations \eqref{eq:lvl_set_mass_derivative_t} and \eqref{eq:lvl_set_mass_derivative_0}, respectively. The fact that the second assertion implies the third one 
is an immediate consequence of the form of the Markov kernel in equation \eqref{eq:markov_kernel_arch}. Applying disintegration proves that the third assertion implies the fourth one, which completes the proof.
\end{proof}
\noindent The previous lemma has the following direct corollary.
\begin{Cor}\label{cor:point.mass}
Let $C \in \mathcal{C}_{ar}^d$ and $\gamma \in \mathcal{P}_{\mathcal{W}_d}$ be the corresponding Williamson measure.
Then $\gamma$ has a point mass if, and only if $\mu_C^{dis}(\mathbb{I}^d) > 0$.
\end{Cor}
As next step we tackle the one-sided versions of the previous corollary 
for the absolutely continuous and the singular components of the Williamson measure.
\begin{Lemma}\label{lem:absolutely_cont}
    Let $C \in \mathcal{C}_{ar}^d$ and $\gamma \in \mathcal{P}_{\mathcal{W}_d}$ be the corresponding Williamson measure. Then $\gamma^{abs}([0,\infty)) > 0$ implies $\mu_C^{abs}(\mathbb{I}^d) >0$. 
\end{Lemma}
\begin{proof}
We assume that $\gamma^{abs}([0,\infty)) > 0$ and prove that for all $\mathbf{x}$ in a set of 
positive  $\mu_{C^{1:d-1}}$-measure the Markov kernel $K_C(\mathbf{x},\cdot)$ has non-degenerated absolutely continuous component. 
Setting $z_0 := \sup\{z \in [0,\infty) \colon \gamma^{abs}([0,z]) = 0\}$ we have $z_0 \in [0,\infty)$. 
Define the set $\Lambda \subseteq \mathbb{I}^{d-1}$ by 
$$
\Lambda := \left\{\mathbf{x} \in [0,1)^{d-1}\colon 0 < \sum_{i = 1}^{d-1}\varphi(x_i) < \frac{1}{z_0}\right\}.
$$
Then we obviously have $\lambda_{d-1}(\Lambda) > 0$. 
Using the equivalence of $\gamma([0,r))=0$ and $\psi(\frac{1}{r})=0$ for $r \in (0,\infty)$, it 
follows that 
$\frac{1}{z_0} \leq \varphi(0)$ holds. Hence for $\mathbf{x} \in \Lambda$ the density $c^{1:d-1}$
of $C^{1:d-1}$ fulfills 
$c^{1:d-1}(\mathbf{x})>0$, which altogether implies that $\mu_{C^{1:d-1}}(\Lambda) > 0$. 
Therefore, using equation (\ref{eq:kernel_function_G}) it 
 suffices to show that for fixed $\mathbf{x} \in \Lambda$ the function
\begin{equation}\label{eq:def_abs_dis_sing_sub_ker}
y \mapsto H^{abs}_\mathbf{x}(y) = 
\frac{\int_{I_y}t^{d-1}\mathrm{d}\gamma^{abs}(t)}{\int_{I_1}t^{d-1}\mathrm{d}\gamma(t)}
\end{equation}
is non-degenerated and absolutely continuous on $\mathbb{I}$, which can be done as follows:
For $\mathbf{x} \in \Lambda$ the construction of $z_0$ implies that for $y_0=f^0(\mathbf{x})$ we have 
$$
\frac{1}{\sum_{i = 1}^{d-1}\varphi(x_i) + \varphi(y_0)} \leq \frac{1}{\varphi(0)} \leq  z_0 < \frac{1}{\sum_{i = 1}^{d-1}\varphi(x_i)} = \frac{1}{\sum_{i = 1}^{d-1}\varphi(x_i) + \varphi(1)},
$$
which, using the fact that $t^{d-1}>0$ for every $t \in (0,\infty)$, yields that $H^{abs}_\mathbf{x}$ is non-degenerated;  
$H^{abs}_\mathbf{x}$ is obviously non-negative, non-decreasing, and continuous on $\mathbb{I}$. \\
Moreover the function $z \mapsto \int_{[0,z]} t^{d-1} \mathrm{d}\gamma^{abs}(t)$ is absolutely continuous 
and non-decreasing on every compact interval of the form $[0,a]$. 
Considering that both, in the strict and in the non-strict case, we have that the mapping  
$y \mapsto \frac{1}{\sum_{i = 1}^{d-1}\varphi(x_i) + \varphi(y)}$ is absolutely continuous and non-decreasing 
on $\mathbb{I}$, according to \cite[Proposition 129]{pap2002} the composition $H^{abs}_\mathbf{x}$ 
is absolutely continuous 
too. Finally, equation \eqref{eq:kernel_function_G_abs_dis_sing} implies that $K_C(\mathbf{x},\cdot)$ has a non-degenerated absolutely continuous component, which completes the proof.
% considering the non-strict case and using that the function $\varphi$ is convex yields absolute continuity of $\varphi$ on $\mathbb{I}$ and thus, the function $y \mapsto \frac{1}{\sum_{i = 1}^{d-1}\varphi(x_i) + \varphi(y)}$ is absolutely continuous on $\mathbb{I}$. Considering the case when $\varphi$ is strict, applying the same chain of arguments as before yields that  $y \mapsto \frac{1}{\sum_{i = 1}^{d-1}\varphi(x_i) + \varphi(y)}$ is absolutely continuous on $[\varepsilon,1]$ for every $0 < \varepsilon < 1$. Since
%\begin{align*}
%\int_{[0,t]}-\frac{D^+\varphi(s)}{(\sum_{i=1}^{d-1}\varphi(x_i) + \varphi(s))^2} \mathrm{d}\lambda(s) &= \lim_{\varepsilon \downarrow 0}\int_{[\varepsilon,t]}-\frac{D^+\varphi(s)}{(\sum_{i=1}^{d-1}\varphi(x_i) + \varphi(s))^2} \mathrm{d}\lambda(s) \\&=
%\frac{1}{\sum_{i = 1}^{d-1}\varphi(x_i) + \varphi(t)} - \lim_{\varepsilon \downarrow 0}\frac{1}{\sum_{i = 1}^{d-1}\varphi(x_i) + \varphi(\varepsilon)} \\&=
%\frac{1}{\sum_{i = 1}^{d-1}\varphi(x_i) + \varphi(t)}
%\end{align*}
%holds for every $t \in [0,1]$,  the function $y \mapsto \frac{1}{\sum_{i = 1}^{d-1}\varphi(x_i) + \varphi(y)}$ is indeed absolutely continuous on $\mathbb{I}$.
\end{proof}
As next step we prove the analogous assertion for the singular component of the measure $\gamma \in \mathcal{P}_{\mathcal{W}_d}$. 
\begin{Lemma}\label{lem:sing}
    Let $C \in \mathcal{C}_{ar}^d$ and $\gamma \in \mathcal{P}_{\mathcal{W}_d}$ be the corresponding Williamson measure. Then $\gamma^{sing}([0,\infty)) > 0$ implies $\mu_C^{sing}(\mathbb{I}^d) >0.$ 
\end{Lemma}
\begin{proof}
We assume $\gamma^{sing}([0,\infty)) > 0$ and proceed in the same manner as for the absolutely continuous 
component. Define the set $\Lambda$ in the same manner as in the proof of Lemma \ref{lem:absolutely_cont} 
only replacing $\gamma^{abs}$ by $\gamma^{sing}$ in the definition of $z_0$.
It suffices to show that for fixed $\mathbf{x} \in \Lambda$ the function
\begin{equation}
y \mapsto H^{sing}_\mathbf{x}(y) = \frac{\int_{I_y}t^{d-1}\mathrm{d}\gamma^{sing}(t)}{\int_{I_1}t^{d-1}\mathrm{d}\gamma(t)}
\end{equation}
is non-degenerated and singular on $\mathbb{I}$, which according to equation \eqref{eq:kernel_function_G_abs_dis_sing} implies that the Markov kernel $K_C(\mathbf{x},\cdot)$ has 
a non-degenerated singular component. \\
The construction of $\Lambda$ implies that $H^{sing}_\mathbf{x}$ is non-degenerated. 
Moreover $H^{sing}_\mathbf{x}$ is 
continuous since discontinuity points would correspond to point masses of $\gamma^{sing}$, of which 
by assumption there are none. \\
In order to prove singularity of $H^{sing}_\mathbf{x}$ let $\mathbf{x} \in \Lambda$ be arbitrary but fixed and 
proceed as follows:  
Define the $\sigma$-finite measure $m$ on $\mathcal{B}([0,\infty))$ by
$$
m(E) := \int_E t^{d-1} \mathrm{d}\gamma^{sing}(t). 
$$
Then there exists some set $O \in \mathcal{B}([0,\infty))$ with $m([0,\infty) \setminus O) = 0$ and 
$\lambda(O) = 0$. Letting $G_m \colon I_1 \rightarrow [0,\infty)$ denote the induced measure-generating function
(restricted to the set $I_1$), defined by $G_m(z) := m([0,z])$, 
singularity of $m$ implies that $G_m'(z) = 0$ for $\lambda$-almost every $z \in I_1$, i.e., setting
$$
F := \{z \in I_1 \colon G_m' \text{ exists and } G_m'(z) = 0\}
$$
we have $\lambda(F) = \lambda(I_1)$. 
The function $g: \mathbb{I} \rightarrow \left[\frac{1}{\sum_{k=1}^{d-1}\varphi(x_k) + \varphi(0)},\frac{1}{\sum_{k=1}^{d-1}\varphi(x_k)}\right] =: I_{0,1}$, defined by 
$g(y) := \frac{1}{\sum_{i=1}^{d-1}\varphi(x_i) + \varphi(y)}$ obviously is
strictly increasing, continuous and bijective. Let $h$ denote its inverse and set
$$
\Upsilon := \left\{y \in \mathbb{I} \colon g(y) \in F\right\}=g^{-1}(F) \in \mathcal{B}(\mathbb{I}).
$$
Then for every $z \in I_{0,1}$ we have $h(z) = \psi(\frac{1}{z}-\sum_{i = 1}^{d-1}\varphi(x_i))$ 
as well as $h(F) = \Upsilon$.
Using convexity of $\psi$ yields that $h$ is locally Lipschitz continuous and therefore locally 
absolutely continuous. Applying that locally absolutely continuous functions map 
sets of $\lambda$-measure $0$ to sets of $\lambda$-measure $0$  \cite[Theorem 3.41]{leoni2024} 
implies that $\lambda(\mathbb{I} \setminus \Upsilon) = 0$ holds.
For fixed $y \in \Upsilon \cap (0,1)$, using equation \eqref{eq:def_abs_dis_sing_sub_ker} and applying the chain rule 
yields  
\begin{align*}
(H^{sing})'(y) &= \frac{1}{c}\frac{\mathrm{d}}{\mathrm{d}y}G_m\bigg(\frac{1}{\sum_{i=1}^{d-1}\varphi(x_i) + \varphi(y)}\bigg) \\&= \frac{1}{c}G_m'\bigg(\frac{1}{\sum_{i=1}^{d-1}\varphi(x_i) + \varphi(y)}\bigg)\frac{-D^+\varphi(y)}{(\sum_{i=1}^{d-1}\varphi(x_i) + \varphi(y))^2} = 0,
\end{align*}
whereby $c = \int_{I_1}t^{d-1}\mathrm{d}\gamma(t)$.
Since $\lambda(\Upsilon \cap (0,1))=0$ we have shown that $H^{sing}_\mathbf{x}$ has derivative zero 
$\lambda$-almost everywhere in $\mathbb{I}$, which completes the proof. 
\end{proof}
%Proceeding analogously as in the absolutely continuous and singular case, we present an alternative proof of the other direction of Lemma \ref{lem:equivalence_point_mass}.
\begin{Rem}\label{rem:discrete}
\emph{Notice that in the proofs of Lemma \ref{lem:absolutely_cont} and Lemma \ref{lem:sing}
the construction via $z_0$ and $\Lambda$
was key for showing that $H^{abs}_\mathbf{x}$ and $H^{sing}_\mathbf{x}$ are non-degenerated, but not for showing absolute 
continuity and singularity. The proofs showed that for $\mathbf{x} \in \Lambda$ the 
functions $H^{abs}_\mathbf{x}$ and $H^{sing}_\mathbf{x}$ are purely absolutely continuous and singular, respectively. 
For the discrete case Corollary \ref{cor:point.mass} states that 
$\gamma^{dis}([0,\infty)) > 0$ implies $\mu_C^{abs}(\mathbb{I}^d) >0$ but it was not 
explicitly shown that $H_\mathbf{x}^{dis}$ is purely discrete. This observation, however, can be shown analogously to the 
absolutely continuous and singular case. In fact, defining 
the set $\Lambda$ in the same manner as in the proof of Lemma \ref{lem:absolutely_cont} 
only replacing $\gamma^{abs}$ by $\gamma^{dis}$ in the definition of $z_0$ we again have
$\mu_{C^{1:d-1}}(\Lambda) > 0$. Moreover, for $\mathbf{x} \in \Lambda$ considering the definition 
of $H^{dis}_\mathbf{x}$ it is straightforward to show that for $\gamma^{dis} = \sum_{j = 1}^\infty\alpha_j\delta_{z_j}$ 
the function $H^{dis}_\mathbf{x}$ corresponds to the discrete measure with point mass 
$\frac{\alpha_j\,z_j^{d-1}}{\int_{I_1} t^{d-1} \mathrm{d}\gamma}$ in  
$\psi(\frac{1}{z_j}- \sum_{i = 1}^{d-1}\varphi(x_i))$ for every $j \in \mathbb{N}$.
}
\end{Rem}
Combining the afore-mentioned lemmas and corollaries we obtain the main result of this section stating that 
the absolutely continuous/discrete/singular component of $\gamma$ is non-degenerated if, and only if the absolutely continuous/discrete/singular component of the corresponding Archimedean copula $C$ is.
Theorem \ref{thm:main_regularity_result} will also be crucial for proving Theorem 
\ref{thm:typical_archimedean} and Corollary \ref{cor:typical_arch_copula}, 
the main result of Section \ref{section:category} on topologically typical Archimedean copulas.
\begin{theorem}\label{thm:main_regularity_result}
    Let $C \in \mathcal{C}_{ar}^d$ and $\gamma \in \mathcal{P}_{\mathcal{W}_d}$ be its corresponding Williamson measure. Then the following three equivalences hold:
    \begin{description}
        \item[(i)] $\gamma^{dis}([0,\infty)) > 0$ if, and only if $\mu_C^{dis}(\mathbb{I}^d) > 0$.
        \item[(ii)] $\gamma^{abs}([0,\infty)) > 0$ if, and only if $\mu_C^{abs}(\mathbb{I}^d) > 0$.
        \item[(iii)] $\gamma^{sing}([0,\infty)) > 0$ if, and only if $\mu_C^{sing}(\mathbb{I}^d) > 0$.
    \end{description}
\end{theorem}
\begin{proof}
Equivalence (i) has already been stated in Corollary \ref{cor:point.mass}. Moreover, for the assertions 
(ii) and (iii) one implication has already been proved. \\
Assume now that $\gamma^{abs}$ is degenerated. Then $\gamma = \gamma^{dis} + \gamma^{sing}$ and at least one of the 
two components is non-degenerated. Considering $\mathbf{x} \in \mathbb{I}^{d-1} \setminus L_0^{1:d-1}$
with $M(\mathbf{x})<1$, proceeding as in the (last parts of the) proofs of Lemma \ref{lem:absolutely_cont} 
and Lemma \ref{lem:sing}, 
and considering equation (\ref{eq:kernel_function_G_abs_dis_sing}) once more yields
$$
K_C(\mathbf{x},[0,y]) = H^{dis}_\mathbf{x}(y) + H^{sing}_\mathbf{x}(y), \quad y \in \mathbb{I}, 
$$
with $H^{dis}_\mathbf{x}$ being discrete and $H^{sing}_\mathbf{x}$ being singular. 
In other words, $K_C(\mathbf{x},[0,y])$ has degenerated absolutely continuous component.
The fact that $\mu_C^{abs}(\mathbb{I}^d) = 0$ now follows immediately via disintegration. \\
The remaining implication in assertion (iii) can be proved in the same manner.
\end{proof}
Theorem \ref{thm:main_regularity_result} allows to extend \cite[Theorem 5.12]{mult_arch} to equivalences - 
the following result holds:
\begin{Cor}\label{cor:regularity_arch}
    Let $C \in \mathcal{C}_{ar}^d$ and $\gamma \in \mathcal{P}_{\mathcal{W}_d}$ be the 
    corresponding Williamson measure. Then the following equivalences hold:
    \begin{description}
        \item[(i)] $\gamma$ is absolutely continuous if, and only if $\mu_C^{abs}(\mathbb{I}^d) = 1$.
        \item[(ii)] $\gamma$ is discrete if, and only if $\mu_C^{dis}(\mathbb{I}^d) = 1$.
        \item[(iii)] $\gamma$ is singular if, and only if $\mu_C^{sing}(\mathbb{I}^d) = 1$.
   \end{description}
\end{Cor}
\begin{proof}
The first assertion has already been established in \cite{neslehova}. 
Furthermore, sufficiency concerning assertions (ii) and (iii) is part of \cite[Theorem 5.12]{mult_arch} and 
it remains to show necessity. 
Assume that $0 \leq \gamma^{dis}([0,\infty)) <1$ holds. Then we have 
$\gamma^{abs}([0,\infty)) > 0$ or $\gamma^{sing}([0,\infty)) > 0$, so applying Theorem \ref{thm:main_regularity_result} immediately yields that $0 \leq \mu_C^{dis}(\mathbb{I}^d)<1$. 
An analogous argument shows (iii). 
\end{proof}
Viewing Theorem \ref{thm:arch_full_support} and again considering the Lebesgue decomposition of the 
Williamson measure $\gamma$ we can show the following stronger statement.
\begin{theorem}\label{lem:full_supp_meas_full_supp_cop}
    Let $C \in \mathcal{C}_{ar}^d$ and let $\gamma \in \mathcal{P}_{\mathcal{W}_d}$ be its corresponding Williamson measure. Then the following assertions hold
    \begin{itemize}
     \item[(i)] If $\gamma^{abs}$ has full support, then $\mu_C^{abs}$ has full support.
        \item[(ii)] If $\gamma^{dis}$ has full support, then $\mu_C^{dis}$ has full support.
         \item[(iii)] If $\gamma^{sing}$ has full support, then $\mu_C^{sing}$ has full support.
    \end{itemize}
\end{theorem}
\begin{proof}
  To prove the first assertion, let $\mathbf{x} \not\in L_0^{1:d-1}$ with $M(\mathbf{x}) < 1$ be arbitrary but fixed and suppose that $\gamma^{abs}$ has full support, which, in particular implies that $C$ is strict. 
It follows that the function $H^{abs}_\mathbf{x}$ defined according to 
equation \eqref{eq:kernel_function_G_abs_dis_sing}
is strictly increasing and absolutely continuous, so $K_C^{abs}(\mathbf{x},\cdot)$ has support $\mathbb{I}$.
Since $\mathbf{x} \not\in L_0^{1:d-1}$ with $M(\mathbf{x}) < 1$ was arbitrary, disintegration directly yields 
that $\mu_C^{abs}$ has support $\mathbb{I}^d$.\\
The second and the third assertion can be proved analogously.
\end{proof}
The previous results open the door to deriving some first (to a certain extent surprising) results - 
we start with denseness of subclasses of Archimedean copulas with full support. 
\begin{theorem}\label{thm:archimedean_dense}
The following assertions hold:
\begin{description}
    \item[(i)] The family $\{C\in \mathcal{C}_{ar,abs}^d\colon\mathrm{supp}(\mu_C)=\mathbb{I}^d\}$ is dense in $(\mathcal{C}_{ar}^d,d_\infty)$,
    \item[(ii)] The family $\{C\in \mathcal{C}_{ar,dis}^d\colon\mathrm{supp}(\mu_C)=\mathbb{I}^d\}$ is dense in $(\mathcal{C}_{ar}^d,d_\infty)$,
    \item[(iii)] The family $\{C\in \mathcal{C}_{ar,sing}^d\colon\mathrm{supp}(\mu_C)=\mathbb{I}^d\}$ is dense in $(\mathcal{C}_{ar}^d,d_\infty)$.
\end{description}
\end{theorem}
\begin{proof}
We fix $C \in \mathcal{C}_{ar}^d$ and let $\gamma \in \mathcal{P}_{\mathcal{W}_d}$ denote the 
corresponding Williamson measure. 
According to Lemma \ref{lem:dense_discrete_measures} there exists a 
sequence $(\gamma_n)_{n \in \mathbb{N}}$ 
of absolutely continuous/discrete/singular Williamson measures with full support such that $(\gamma_n)_{n \in \mathbb{N}}$ converges weakly to $\gamma$. Applying \cite[Theorem 5.9]{mult_arch} yields that the sequence 
$(C_n)_{n \in \mathbb{N}}$ of corresponding Archimedean copulas converges uniformly to $C$. 
Since each $\gamma_n \in \mathcal{P}_{\mathcal{W}_d}$ has full support and is 
absolutely continuous/discrete/singular, applying Theorem \ref{thm:arch_full_support} and Corollary \ref{cor:regularity_arch} shows that each $C_n$ is absolutely continuous/discrete/singular and that 
$\mathrm{supp}(\mu_C)=\mathbb{I}^d$ holds. This completes the proof since $C \in \mathcal{C}_{ar}^d$ was
arbitrary.
\end{proof}
Using a similar idea of proof the following even stronger version of Theorem \ref{thm:archimedean_dense}, stating
that even the family of Archimedean copulas, whose absolutely continuous, discrete and singular components 
(simultaneously) have full support, is dense, can be shown.
\begin{Cor}
The family $\{C\in \mathcal{C}_{ar}^d\colon\mathrm{supp}(\mu_C^{abs})=\mathrm{supp}(\mu_C^{dis})=\mathrm{supp}(\mu_C^{sing})=\mathbb{I}^d\}$ is dense in $(\mathcal{C}_{ar}^d,d_\infty)$.
\end{Cor}
\begin{proof}
Fix $C \in \mathcal{C}_{ar}^d$ and let $\gamma \in \mathcal{P}_{\mathcal{W}_d}$ denote the 
corresponding Williamson measure. According to Lemma \ref{lem:dense_discrete_measures} there exist sequences 
$(\gamma_n^{(1)})_{n \in \mathbb{N}}$, $(\gamma_n^{(2)})_{n \in \mathbb{N}}$ and $(\gamma_n^{(3)})_{n \in \mathbb{N}}$ of
purely absolutely continuous, discrete and singular Williamson measures having full support which all converge weakly to
$\gamma$. Considering $$\gamma_n := \frac{1}{3}\left(\gamma_n^{(1)} + \gamma_n^{(2)} + \gamma_n^{(3)}\right)$$ for every
$n \in \mathbb{N}$, using Theorem \ref{lem:full_supp_meas_full_supp_cop}, and proceeding as in the proof of 
Theorem \ref{thm:archimedean_dense} the desired result follows. 
\end{proof}
%\clearpage
\section{Regularity of the Kendall distribution function of Archimedean copulas}\label{sec:kendall}
In this short section we revisit the interplay between the Williamson measure $\gamma$ and the Kendall distribution 
function $F_K^d$  of the corresponding copula $C\in \mathcal{C}_{ar}^d$ and show, loosely speaking, that 
regularity of $C$ (and hence of $\gamma$) goes hand in hand with regularity of $F_K^d$. 
Note that in the full class of copulas a similar behavior does not hold. In fact, considering $d=2$
and the minimum copula $M$, the Kendall distribution function is given by $F_K^2(t) = t$ 
(see, e.g., \cite[Example 3.9.6.]{dur_princ}), so $M$ is discrete although its corresponding 
Kendall distribution function is absolutely continuous. 

Theorem \ref{thm:lvl_sets_kendall_williamson} shows that the Kendall distribution function $F_K^d$ can 
nicely be re\-presented in terms of $\gamma$ (and the quasi-inverse $\varphi$ of the generator $\psi$). 
The subsequent lemma provides an expression for $\gamma$ via $F_K^d$, which we will use in the sequel:
\begin{Lemma}\label{lem_gamma.from.K}
Consider $C \in \mathcal{C}_{ar}^d$ and let $\psi$, $\gamma$, and $F_K^d$ denote the 
corresponding generator, Williamson measure and Kendall distribution function, respectively. 
Then the following identity holds for all $z \in [0,\infty)$:
      \begin{equation}\label{eq:con_kendall_williamson}
     \gamma([0,z]) = \begin{cases}
        F_K^d(\psi(\frac{1}{z})),& \text{ if } z \in [\frac{1}{\varphi(0)},\infty)\\ 
        0,& \text{ if } z \in [0,\frac{1}{\varphi(0)}). 
     \end{cases}
    \end{equation}
Moreover, the non-decreasing transformation $h: [0,\infty] \rightarrow \mathbb{I}$, defined by $h(z)=\psi(\frac{1}{z})$ 
fulfills $h(0)=0$ and is non-singular in the sense that for every $F \in \mathcal{B}([0,\infty))$ we have that 
$\lambda(F)=0$ implies $\lambda^h(F)=0$.   
\end{Lemma}
\begin{proof}
As already mentioned in Section 2 for every $r \in (0,\infty)$
we have $\gamma([0,r))=0$ if, and only if $\psi(\frac{1}{r})=0$. This directly yields $\gamma([0,z])=0$
for every $z \in [0,\frac{1}{\varphi(0)})$. \\
On the other hand, for $z \geq \frac{1}{\varphi(0)}$ there exists a unique $t \in \mathbb{I}$ with 
$\varphi(t)=\frac{1}{z}$, and this $t$ coincides with $\psi(\frac{1}{z})$. Applying 
equation \eqref{eq:kendall_williamson} the desired identity $\gamma([0,z]) = F_K^d(\psi(\frac{1}{z}))$ follows. \\
The second assertion is a direct consequence of the facts that (i) $\psi$ is a generator and 
(ii) the mapping $z \mapsto \frac{1}{z}$ is Lipschitz continuous on every interval of the form $[a,1]$ with $a>0$.   
\end{proof}
As next step we use the previous lemma to characterize regularity of the Kendall distribution functions in terms of 
regularity of the Williamson measure. Having in mind the results from the previous section the following 
statements are not surprising.
\begin{theorem}\label{thm:reg_williamson_kendall}
    Let $C \in \mathcal{C}_{ar}^d$, $\gamma \in \mathcal{P}_{\mathcal{W}_d}$ be the corresponding Williamson measure, 
    $F_K^d$ be the Kendall distribution function of $C$ and $\kappa_{F_K^d}$ the probability measure
    corresponding to $F_K^d$. Then the following equivalences hold:
    \begin{itemize}
        \item[(i)] $\gamma$ is absolutely continuous if, and only if $\kappa_{F_K^d}$is absolutely continuous.
         \item[(ii)] $\gamma$ is discrete if, and only if $\kappa_{F_K^d}$ is discrete.
          \item[(iii)] $\gamma$ is singular if, and only if $\kappa_{F_K^d}$ is singular.
    \end{itemize}
\end{theorem}
\begin{proof}
We start with a general observation concerning the function $h$ from Lemma \ref{lem_gamma.from.K} which we will use 
for proving the equivalences. For every $t \in \mathbb{I}$ using Theorem \ref{thm:lvl_sets_kendall_williamson}
we have 
\begin{align}\label{eq:gammah}
\gamma^h([0,t]) &= \gamma\left(\left\{z \in \mathbb{I}: h(z) \leq t \right\}\right) = 
\gamma\left(\left\{z \in \mathbb{I}: \frac{1}{z} \geq \varphi(t) \right\}\right) \nonumber \\
&= \gamma\left(\left[0, \frac{1}{\varphi(t)} \right]\right) = F_K^d(t) = \kappa_{F_K^d}([0,t]).
\end{align}
In other words, the measure $\kappa_{F_K^d}$ coincides with the push-forward of $\gamma$ via $h$. \\ 
(i) Suppose now that $\gamma$ is absolutely continuous. Then, using the fact that $h$ is non-singular 
it follows that $\gamma^h$ is absolutely continuous on $\mathcal{B}((0,1])$. Moreover
in this case we have 
$$
\gamma^h(\{0\}) = \gamma\left(\left[0, \frac{1}{\varphi(0)} \right]\right) = 
\gamma\left(\left[0, \frac{1}{\varphi(0)} \right)\right)=0,
$$
so altogether $\gamma^h=\kappa_{F_K^d}$ is absolutely continuous. \\
To prove the reverse implication we may proceed in the same manner working with the injective function
$g: \mathbb{I} \rightarrow [0,\infty]$, defined by $g(z)=\frac{1}{\varphi(z)}$ for 
every $z \in \mathbb{I}$, instead of the function $h$,
and use the fact that $\kappa_{F_K^d}^g=\gamma$.  \\
(ii) Immediate consequence of equation (\ref{eq:kendall_williamson}) and the fact that 
the afore-mentioned function $g$ is injective.  \\
(iii) If $\gamma$ is singular then using equation (\ref{eq:gammah}) it follows that $\kappa_{F_K^d}$ 
has no point-masses. Moreover, letting $H_\gamma$ denote the distribution function of $\gamma$, singularity 
implies that the set $E \in \mathcal{B}([0,\infty))$, defined by
$$
 E := \{z \in [0,\infty) \colon F_\gamma'(z) \text{ exists and } F_\gamma'(z) \neq 0\}
$$
fulfills $\lambda(E) = 0$. 
Working with the fact that the function $g,h$ are locally Lip\-schitz continuous it follows (see, e.g., 
\cite[Theorem 3.41]{leoni2024}) that the set $\Upsilon := \left\{t \in (0,1)\colon g(t) \in E\right\}$ 
fulfills $\lambda(\Upsilon)=0$, implying that
    $$
    \frac{\mathrm{d}}{\mathrm{d}t}F_K^d(t) = 
    -F_\gamma'\left(\frac{1}{\varphi(t)}\right)\frac{\varphi'(t)}{\varphi(t)^2} = 0
    $$
    holds for $\lambda$-almost every $t \in \mathbb{I}$. This shows that $\kappa_{F_K^d}$ is singular. \\
   The reverse implication can be shown using analogous arguments and working with equation 
   (\ref{eq:con_kendall_williamson}).  
\end{proof}
Applying the results from the previous section yields the following: 
\begin{Cor}
Let $C \in \mathcal{C}_{ar}^d$ and $\gamma$ and $F_K^d$ denote the corresponding Williamson measure and 
Kendall distribution function, respectively. Then the following equivalences hold:
\begin{itemize}
\item[(i)] $C$ is absolutely continuous if, and only if $\kappa_{F_K^d}$ is absolutely continuous.
\item[(ii)] $C$ is discrete if, and only if $\kappa_{F_K^d}$ is discrete.
\item[(iii)] $C$ is singular if, and only if $\kappa_{F_K^d}$ is singular.
\end{itemize}
\end{Cor}
%\clearpage
\section{Derivatives of multivariate Archimedean copulas}\label{section:derivatives}
In \cite{dietrich2024} it was shown that `derivatives of copulas have to be handled with care', a statement 
that was underlined by showing the existence of a dense class of bivariate copulas $C$ of the following type: 
for $\lambda$-almost every $x \in (0,1)$ the partial derivative $\partial_1C(x,y)$ 
does not exist on a dense subset of $y \in (0,1)$. 
Motivated by these results in this section we tackle the question if - in the sense of the derivative - 
analogously pathological copulas also exist in the Archimedean family. \\
It is well known (and follows directly from the properties of generators) 
that every $C \in \mathcal{C}_{ar}^d$ is continuously differentiable up to order $d-2$ (see \cite{neslehova}).
Moreover, as mentioned in the introduction, for every fixed $y \in (0,1)$ the partial derivative 
$\partial_1...\partial_{d-1}C(\mathbf{x},y)$ exists for $\lambda_{d-1}$-almost every 
$\mathbf{x} = (x_1,...,x_{d-1}) \in \mathbb{I}^{d-1}$. \\ 
We will prove in this section that it is not possible to go much further. As main result 
we will show that $\mathcal{C}_{ar}^d$ contains copulas $C$ of the following pathological type: 
there exists some set $\Lambda \in \mathcal{B}(\mathbb{I}^{d-1})$ with 
$\lambda_{d-1}(\Lambda) = 1$ such that for every $\mathbf{x} \in \Lambda$ the $(d-1)$-st order partial 
derivative $\partial_1\partial_2...\partial_{d-1}C(\mathbf{x},y)$ does not exist
on a dense set of $y \in (0,1)$. The family of all such $d$-dimensional Archimedean copulas will be 
denoted by $\mathcal{C}_{ar,\mathcal{Q}}^d$.  
As in \cite{dietrich2024} we first study the class $\mathcal{C}_{ar,p}^d$ of all $d$-dimensional
Archimedean copulas fulfilling that for every $\mathbf{x}$ from a set $\Lambda \in \mathcal{B}(\mathbb{I}^{d-1})$ with 
$\mu_{C^{1:d-1}}(\Lambda) > 0$ there exists some $y := y_\mathbf{x} \in (0,1)$ such that the $(d-1)$-st order 
partial derivative $\partial_1\partial_2...\partial_{d-1}C(\mathbf{x},y)$ does not exist.\\
We start with the following simple example of an element in $\mathcal{C}_{ar,p}^3$ illustrating the 
ideas for the general setting. 
\begin{Ex}\label{ex:non_diff}
\normalfont
    Considering the Williamson measure $\gamma = \frac{32}{49} \, \delta_{\frac{1}{8}} + \frac{17}{49} \, \delta_{2}$ 
    the Williamson $3$-transform (\ref{eq:williamson_transform}) yields the non-strict generator 
    $\psi \colon [0,\infty) \rightarrow \mathbb{I}$, given by 
    $$
    \psi(z) = \begin{cases}
			 \frac{137z^2-152z+98}{98}, & \text{ if } z \in [0,\frac{1}{2}], \\
			 \frac{32}{49}(1-\frac{z}{8})^2, & \text{ if } z \in (\frac{1}{2},8], \\
			0, & \text{ otherwise.}
		\end{cases}
    $$
    Theorem \ref{thm:lvl_sets_kendall_williamson} implies that the graphs of the functions $f^{a}$ 
    with $a=0 = \psi(8)$ and $a=\psi(2)=\frac{225}{392}$ have mass $\frac{32}{49}$ and $\frac{17}{49}$, respectively.  
    We will focus on $f^0$ and show that for every $\mathbf{x} \not\in L_0^{1:2}$ with $M(\mathbf{x})<1$ 
    we have that for $y_\mathbf{x} = f^0(\mathbf{x}) \in (0,1)$ the mixed derivative 
     $\partial_1\partial_2C(\mathbf{x},y_\mathbf{x})$ does not 
    exist. Notice that this suffices since we have $\mu_{C^{1:2}}(L_0^{1:2} \cup M^{-1}(\{1\}))=0$. \\
    Fix $\mathbf{x}_0 = (x_{0,1},x_{0,2}) \not\in (L_0^{1:2} \cup M^{-1}(\{1\}))$ 
    and consider $y := y_{\mathbf{x}_0} := f^0(\mathbf{x}_0) \in (0,1)$. 
    Then applying change of coordinates, using the fact that the function $-\psi'$ is convex and that we
    have $\varphi(y)>0$,
    together with \cite[Theorem 3.7.4]{kannan1996} and \cite[Appendix C]{Pollard} yields that
    \begin{align*}
        \int_{[0,x_{0,1}]}K_C(s_1,s_2,[0,y])&c^{1:2}(s_1,s_2) \mathrm{d}\lambda(s_1) \\&=
        \varphi'(s_2)\int_{[0,x_{0,1}]} \varphi'(s_1)D^-\psi'(\varphi(s_1) + \varphi(s_2) + \varphi(y)) \mathrm{d}\lambda(s_1) \\&=
        \varphi'(s_2)\psi'(\varphi(x_{0,1}) + \varphi(s_2) + \varphi(y)).
    \end{align*}
    Using disintegration and differentiability of $C$ (of order $1$) it follows that 
    $$
    \partial_2 C(\mathbf{x}_0,y) = \int_{[0,x_{0,1}]}K_C(s_1,x_{0,2},[0,y]) c^{1:2}(s_1,x_{0,2}) \mathrm{d}\lambda(s_1).
    $$
     Applying Lemma \ref{lem:function_G}, for $s_1 > x_{0,1}$ we get 
    \begin{align*}
    K_C(s_1,x_{0,2},[0,y]) c^{1:2}(s_1,x_{0,2}) &= \varphi'(s_1)\varphi'(x_{0,2})D^-\psi'(\varphi(s_1) + \varphi(x_{0,2}) + \varphi(y)) \\&=
    \varphi'(s_1)\varphi'(x_{0,2})\bigg(\frac{1}{49}\mathbf{1}_{(0,\frac{1}{\varphi(s_1) - \varphi(x_{0,1}) + 8}]}\left(\frac{1}{8}\right) \\&\quad+ \frac{136}{49} \mathbf{1}_{(0,\frac{1}{\varphi(s_1) - \varphi(x_{0,1}) + 8}]}\left(2\right) \bigg).
    \end{align*}
    Using the fact that $\frac{1}{\varphi(s_1) - \varphi(x_{0,1}) + 8} > \frac{1}{8}$ and calculating 
    the right-hand partial derivative $\partial_1^+\partial_2 C(\mathbf{x}_0,y)$ it follows that
    \begin{align*}
    \partial_1^+\partial_2 C(\mathbf{x}_0,y) &= \lim_{h \downarrow 0}\frac{1}{h} \int_{(x_{0,1},x_{0,1}+h]}K_C(s_1,x_{0,2},[0,y]) c^{1:2}(s_1,x_{0,2}) \mathrm{d}\lambda(s_1) \\&=
    \frac{\varphi'(x_{0,2})}{49}\lim_{h \downarrow 0}\frac{1}{h}\int_{(x_{0,1},\,x_{0,1}+h]} \varphi'(s_1) \mathrm{d}\lambda(s_1) \\&=
    \frac{\varphi'(x_{0,1})\varphi'(x_{0,2})}{49} > 0.
    \end{align*}
    On the other hand, using the fact that for $s_1<x_{0,1}$ we have 
    $\frac{1}{\varphi(s_1) - \varphi(x_{0,1}) + 8} < \frac{1}{8}$, proceeding analogously 
    for the left-hand partial derivative $\partial_1^-\partial_2 C(\mathbf{x}_0,y) $ we obtain 
    \begin{align*}
    \partial_1^-\partial_2 C(\mathbf{x}_0,y) &= \lim_{h \downarrow 0}\frac{1}{h} \int_{(x_{0,1}-h,\,x_{0,1}]}K_C(s_1,x_{0,2},[0,y]) c^{1:2}(s_1,x_{0,2}) \mathrm{d}\lambda(s_1) \\&= 0.
    \end{align*}
    This shows $\partial_1^+\partial_2C(\mathbf{x}_0,y) \neq \partial_1^-\partial_2C(\mathbf{x}_0,y)$, implying 
    that $\partial_1\partial_2C(\mathbf{x}_0,y)$ does not exist. 
    The graphs of the functions $f^a$ with $a \in \{0,\frac{225}{392}\}$ are depicted in Figure \ref{fig:3d_lvl_sets}.        
    These contain the points $(\mathbf{x},y) \in \mathbb{I}^3$ where the 
    derivative $\partial_1\partial_2C(\mathbf{x},y)$ does not exist. 
\end{Ex}
\begin{figure}[!ht]
	\centering
\includegraphics[width=1\textwidth]{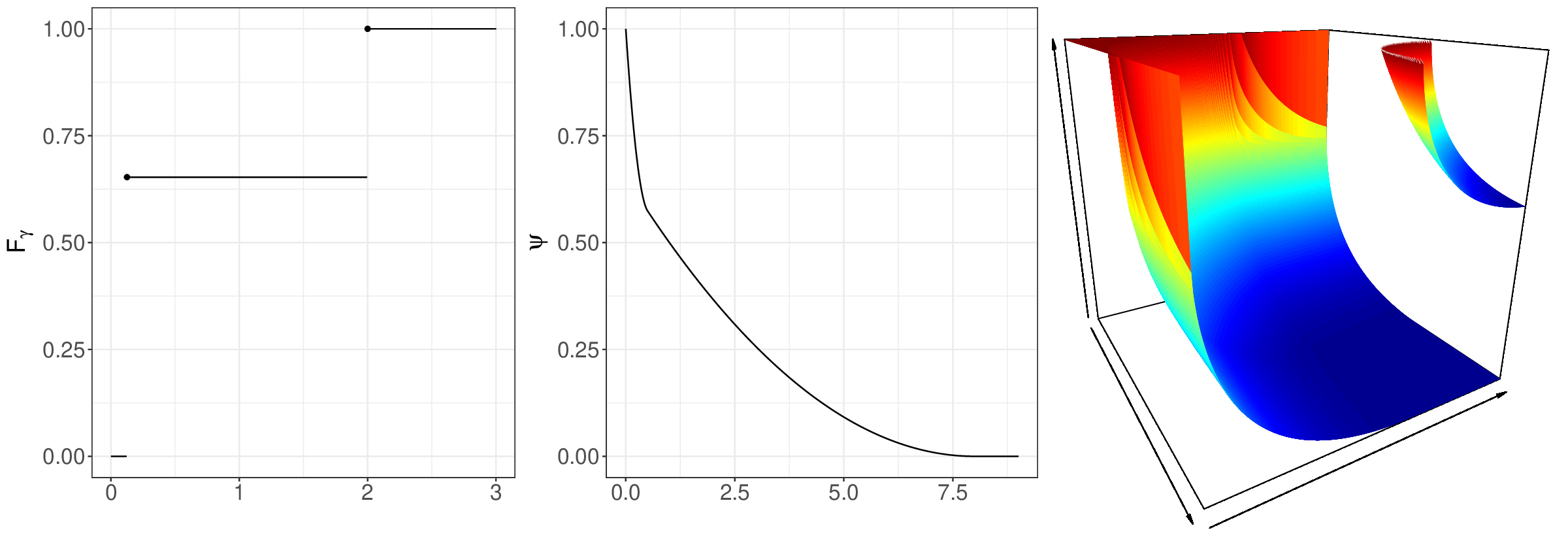}
	\caption{Plots of the distribution function $F_\gamma$ (left panel) of the 
	Williamson measure $\gamma$ considered in Example \ref{ex:non_diff}, the associated generator $\psi$ (middle) 
	and the graphs of the functions $f^a$ with $a \in \{0,\frac{225}{392}\}$ of the corresponding 
	Archimedean copula $C$ (right panel).}
\label{fig:3d_lvl_sets}
\end{figure}
We now focus on the class $\mathcal{C}_{ar,\mathcal{Q}}^d$ and show that it is non-empty.
\begin{theorem}[Non-differentiability on dense subset]\label{thm:Archimedean_non_diff}
There are copulas $C \in \mathcal{C}_{ar}^d$ with the following property: there is some set 
$\Lambda \in \mathcal{B}(\mathbb{I}^{d-1})$ with $\lambda_{d-1}(\Lambda) = 1$ such that for 
every $\mathbf{x} \in \Lambda$ it holds that $\partial_1...\partial_{d-1}C(\mathbf{x},y)$ does not exist
for a dense set of $y \in \mathbb{I}$. \\
The same result still holds for any ordering of the mixed partial derivatives of order $(d-1)$. 
\end{theorem}
\begin{proof}
We prove the result for $d \geq 3$, the case $d = 2$ can be proved similarly.
Suppose that $Q=\{q_1,q_2,...\} \subseteq (0,\infty)$ is dense in $[0,\infty)$, that 
$\alpha_1, \alpha_2, ... \in (0,1)$ fulfill $\sum_{i = 1}^\infty\alpha_i = 1$, and that 
$$
\gamma := \sum_{i \in \mathbb{N}}\alpha_i\delta_{q_i} \in \mathcal{P}_{\mathcal{W}_d}.
$$
holds. For ways to construct Williamson measures of this type see, e.g., \cite[Lemma Appendix B.2]{mult_arch}.
Letting $\psi$ denote the induced strict generator, applying Lemma \ref{lem:function_G} we have
\begin{align}\label{eq:left_hand_der}
\nonumber D^{-}\psi^{(d-2)}(z) &= (-1)^{d-1}(d-1)!\int_{[0,\frac{1}{z}]}t^{d-1} \mathrm{d}\gamma(t) \\&= (-1)^{d-1}(d-1)!\sum_{i \in \mathbb{N}}\alpha_iq_i^{d-1}\mathbf{1}_{(0,\frac{1}{z}]}(q_i)
\end{align}
for every $z > 0$. Moreover, applying change of coordinates and the fact that $(-1)^{d-2}\psi^{(d-2)}$ is convex 
together with \cite[Theorem 3.7.4]{kannan1996} and \cite[Appendix C]{Pollard} as well as using the explicit form of the marginal density given in equation \eqref{eq:density} yields that
\begin{align*}
\int_{[0,x_1]} &K_C(s_1,\mathbf{x}_{2:d-1},[0,y])c^{1:d-1}(s_{1},\mathbf{x}_{2:d-1})\mathrm{d}\lambda(s_{1}) \\&=
\prod_{j = 2}^{d-1}\varphi'(x_j)\int_{[0,x_1]} D^{-}\psi^{(d-2)}\left(\varphi(s_1) + \sum_{i=2}^{d-1}\varphi(x_i) + \varphi(y)\right)\varphi'(s_1)\mathrm{d}\lambda(s_{1}) \\&=
\prod_{j = 2}^{d-1}\varphi'(x_j)\psi^{(d-2)}\left(\sum_{i=1}^{d-1}\varphi(x_i) + \varphi(y)\right).
\end{align*}
Using disintegration and differentiability of $C$ (of order $d-2$) it follows that
$$
\partial_2...\partial_{d-1}C(\mathbf{x}_{1:d-1},y) = \int_{[0,x_1]}K_C(s,\mathbf{x}_{2:d-1},[0,y])c^{1:d-1}(s,\mathbf{x}_{2:d-1}) \mathrm{d}\lambda(s)
$$
holds for all $\mathbf{x} \in (0,1)^{d-1}$ and $y \in (0,1)$.\\
Fix $\mathbf{z} \in (0,1)^{d-1}$ with $z_{k} \in (0,1) \setminus \{\psi(\frac{1}{q_j}): j\in \mathbb{N}\}$ 
for every $k \in \{1,...,d-1\}$ and set $y_i = f^{\psi(\frac{1}{q_i})}(\mathbf{z})$ for all
$i \in \mathbb{N}$ with $\frac{1}{q_i}  > \sum_{k = 1}^{d-1}\varphi(z_{k})$. Notice that in this case
we have $\mathbf{z} \in [C^{1:d-1}]_{\psi\left(\frac{1}{q_i}\right)}$, so $y_i$ is well-defined.
Working with equation \eqref{eq:left_hand_der} yields
\begin{align*}
K_C(s,\mathbf{z}_{2:d-1}&,[0,y_i])c^{1:d-1}(s,\mathbf{z}_{2:d-1}) \\&= \prod_{k = 2}^{d-1}\varphi'(z_k)\varphi'(s)(-1)^{d-1}(d-1)!\sum_{j \in \mathbb{N}}\alpha_j q_j^{d-1}\mathbf{1}_{(0,\frac{1}{\varphi(s) + \sum_{k=2}^{d-1}\varphi(z_k) + \varphi(y_i)}]}(q_j) \\&= 
\underbrace{\prod_{k = 2}^{d-1}\varphi'(z_k)\varphi'(s)(-1)^{d-1}(d-1)!\sum_{j \neq i}\alpha_j q_j^{d-1}\mathbf{1}_{(0,\frac{1}{\varphi(s) + \sum_{k=2}^{d-1}\varphi(z_k) + \varphi(y_i)}]}(q_j)}_{I(s)} \\&\quad+ \underbrace{\prod_{k = 2}^{d-1}\varphi'(z_k)\varphi'(s)(-1)^{d-1}(d-1)!\,\alpha_i q_i^{d-1}\mathbf{1}_{(0,\frac{1}{\varphi(s) + \sum_{k=2}^{d-1}\varphi(z_k) + \varphi(y_i)}]}(q_i)}_{II(s)},
\end{align*}
for fixed $s \in (0,1)$. Since the mapping $i \mapsto q_i$ is injective, working with the fact that $\varphi$ is continuous and decreasing yields that the function $s \mapsto I(s)$ is continuous at $z_1$ and thus the function
$
t \mapsto \int_{[0,t]} I(s) \mathrm{d}\lambda(s)
$
is differentiable at $z_1$. \\
On the other hand, using the fact that $\varphi'$ is continuous and calculating the right-hand 
derivative of the function $t \mapsto \int_{[0,t]} II(s) \mathrm{d}\lambda(s)$ at $z_1$ 
we obtain 
\begin{align*}
    \lim_{h \downarrow 0}\frac{1}{h}\int_{[z_{1},z_{1}+h]}II(s)\mathrm{d}\lambda(s) &= 
    \prod_{k=2}^{d-1}\varphi'(z_k)(-1)^{d-1}(d-1)!\,\alpha_iq_i^{d-1}\cdot\\ & \quad \quad
     \lim_{h \downarrow 0}\frac{1}{h}\int_{[z_{1},z_{1}+h]}\varphi'(s) \mathbf{1}_{(0,\frac{1}{\varphi(s) - \varphi(z_{1}) + \frac{1}{q_i}}]}(q_i)\mathrm{d}\lambda(s) \\&=
    \prod_{k = 1}^{d-1}\varphi'(z_k)(-1)^{d-1}(d-1)!\,\alpha_i q_i^{d-1} > 0.
\end{align*}
Finally, proceeding analogously for the left-hand derivative yields
\begin{align*}
\lim_{h \downarrow 0}-\frac{1}{h}\int_{[z_{1}-h,z_{1}]}II(s)\mathrm{d}\lambda(s) &= 0.
\end{align*}
In other words, the function $s \mapsto \partial_2...\partial_{d-1}C(s,\mathbf{z}_{2:d-1},y_i)$ is not 
differentiable in $z_1$. Since the set $$\left\{f^{\psi(\frac{1}{q_j})}(\mathbf{z}) \colon \frac{1}{q_j} > \sum_{k = 1}^{d-1}\varphi(z_k), j\in\mathbb{N}\right\}$$
 obviously is dense in $\mathbb{I}$, the main result is proved. \\
The second assertion of the theorem follows by proceeding analogously and using Fubini's theorem for changing the 
order of integration/differentiation.  
\end{proof}
The next result shows that elements of $\mathcal{C}_{ar,\mathcal{Q}}^d$ can be found in every open ball of the 
compact metric space $(\mathcal{C}_{ar}^d,d_\infty)$.
\begin{Cor}
 The set $\mathcal{C}_{ar,\mathcal{Q}}^d$ is dense in $(\mathcal{C}_{ar}^d,d_\infty)$.
\end{Cor}
\begin{proof}
Fix $C \in \mathcal{C}_{ar}^d$ and let $\gamma \in \mathcal{P}_{\mathcal{W}_d}$ be its corresponding 
Williamson measure. Then according to Lemma \ref{lem:dense_discrete_measures} there exists a sequence 
$(\gamma_n)_{n \in \mathbb{N}}$ of discrete Williamson measures with full support 
which converges weakly to $\gamma$. Using Theorem \ref{thm:Archimedean_non_diff} every $\gamma_n$ 
induces a copula $C_n \in \mathcal{C}_{ar,\mathcal{Q}}^d$, so applying \cite[Theorem 5.9]{mult_arch} 
or Lemma \ref{lem:hom_archimedean} yields the desired result.
\end{proof}
We will show in the next section that elements of $\mathcal{C}_{ar,\mathcal{Q}}^d$ (in fact even elements of 
$\mathcal{C}_{ar,p}^d$) are very atypical in the sense of Baire categories. 
Before focusing on topological sizes of subsets of $\mathcal{C}_{ar}^d$, however, we conclude this section 
by showing that $\mathcal{C}_{ar,p}^d$ exclusively consists of Archimedean copulas with 
non degenerated discrete component.     
\begin{theorem}\label{thm:equiv_discrete_path}
    For $C \in \mathcal{C}_{ar}^d$ the following two assertions are equivalent: 
    \begin{enumerate}
    \item $\mu_C^{dis}(\mathbb{I}^d) > 0$, i.e., $C$ has non degenerated discrete component.
    \item $C\in \mathcal{C}_{ar,p}^d$.
    \end{enumerate}
\end{theorem}
\begin{proof}
We consider the case $d \geq 3$, the case $d = 2$ is technically simpler and can be proved analogously. \\
(i) Fix $C \in \mathcal{C}_{ar}^d$ and assume that $\mu_C^{dis}(\mathbb{I}^d) > 0$ holds. 
Corollary \ref{cor:point.mass} implies that $\gamma$ has a point mass at some $z_0 \in (0,\infty)$. 
Define $t_0 := \psi(\frac{1}{z_0})$, let $\mathbf{x} \in [C^{1:d-1}]_{t_0}$ be arbitrary but fixed, and set $y:=y_\mathbf{x} := f^{t_0}(\mathbf{x})$.
Then proceeding analogously as in the proof of Theorem \ref{thm:Archimedean_non_diff} and using left-continuity of $D^-\psi^{(d-2)}$ on the one hand yields 
\begin{align*}
\partial_1^+\partial_2...&\partial_{d-1}C(\mathbf{x},y) \\&= \prod_{i=2}^{d-1}\varphi'(x_i)\lim_{h \downarrow 0}\frac{1}{h}\int_{[x_1,x_1+h]}D^{-}\psi^{(d-2)}\left(\varphi(s_1) - \varphi(x_1) + \varphi(t_0)\right)\varphi'(s_1)\mathrm{d}\lambda(s_1) \\&=
\prod_{i=1}^{d-1}\varphi'(x_i)D^{-}\psi^{(d-2)}\left(\varphi(t_0)\right)
\end{align*}
and on the other hand we have
\begin{align*}
\partial_1^-\partial_2...\partial_{d-1}&C(\mathbf{x},y)=
\prod_{i=1}^{d-1}\varphi'(x_i)D^{-}\psi^{(d-2)}\left(\varphi(t_0)+\right).
\end{align*}
Lemma \ref{lem:equivalence_point_mass} implies that $\frac{1}{z_0}$ is a discontinuity 
point of $D^-\psi^{(d-2)}$, so 
$$\partial_1^+\partial_{2}...\partial_{d-1}C(\mathbf{x},y) \neq \partial_1^-\partial_{2}...\partial_{d-1}C(\mathbf{x},y)$$
and we have shown that the mapping $s \mapsto \partial_2...\partial_{d-1}C(s,\mathbf{x}_{2:d-1},y)$ is not 
differentiable at $x_1$.
 Using the fact that $\mu_{C^{1:d-1}}([C^{1:d-1}]_{t_0}) >0$ we get $C \in \mathcal{C}_{ar,p}^d$,
which completes the proof of the first implication. \\
(ii) To show the other direction we proceed as follows. If $C \in \mathcal{C}_{ar,p}^d$ then there exists 
some set $\Lambda \in \mathcal{B}(\mathbb{I}^{d-1})$ with $\mu_{C^{1:d-1}}(\Lambda) > 0$ such that for 
every point $\mathbf{x} \in \Lambda$ there exists some $y_\mathbf{x} \in (0,1)$ with $\partial_1^+\partial_2...\partial_{d-1}C(\mathbf{x},y_\mathbf{x}) \neq \partial_1^-\partial_2...\partial_{d-1}C(\mathbf{x},y_\mathbf{x})$. Explicitly calculating these derivatives as before it follows that 
$$
D^{-}\psi^{(d-2)}\left(\sum_{i=1}^{d-1}\varphi(x_i) + \varphi(y_\mathbf{x})\right) \neq 
D^{-}\psi^{(d-2)}\left(\sum_{i=1}^{d-1}\varphi(x_i) + \varphi(y_\mathbf{x})+\right),
$$ 
which, using equation \eqref{eq:markov_kernel_arch} yields $K_C(\mathbf{x},\{y_\mathbf{x}\}) > 0$. 
Applying disintegration and equation \eqref{eq:def_abs_dis_sing_copula} the property $\mu_C^{dis}(\mathbb{I}^d) > 0$ follows.
\end{proof}
\section{Baire category results for multivariate Archimedean copulas}\label{section:category}
Building upon the previous sections on mass distributions and derivatives we now study, which regularity 
properties topologically typical Archimedean copulas exhibit. Doing so, we first 
derive Baire category results in the family $(\mathcal{P}_{\mathcal{W}_d},\tau_w)$ of all $d$-Williamson measures 
and then translate those results to $\mathcal{C}_{ar}^d$ by using the fact that homeomorphisms preserve 
Baire categories (see, e.g., \cite[Lemma A1]{dietrich2024}). \\
Recall that the space of all Archimedean copulas $(\mathcal{C}_{ar}^d,d_\infty)$ is not complete, so it is 
a priori not clear, that $\mathcal{C}_{ar}^d$ is not of first Baire category in itself. 
The subsequent auxiliary lemmas provide the basis for the Baire category results to follow. 
As mentioned in Section \ref{section:arch_cop}, the family $\mathcal{P}_{nor}^d$ endowed with $\tau_w$
is complete and hence closed. 
\begin{Lemma}
The family of all $d$-Williamson-measures $\mathcal{P}_{\mathcal{W}_d}$ is dense in $(\mathcal{P}_{nor}^d,\tau_w)$.
\end{Lemma}
\begin{proof}
    First of all notice that every $\gamma \in \mathcal{P}_{nor}^d$ fulfills $\gamma(\{0\}) \leq \frac{1}{2}$.
    Since for $\gamma(\{0\})=0$ we have $\gamma \in \mathcal{P}_{\mathcal{W}_d}$ it suffices 
    to consider the case $\gamma(\{0\})   \in (0,\frac{1}{2}]$ and to show that there exists 
    some sequence in $\mathcal{P}_{\mathcal{W}_d}$ converging weakly to $\gamma$. \\
    (i) Suppose that $0<\gamma(\{0\}) < \frac{1}{2}$ holds. Then we have $\gamma(\mathbb{I}) > \frac{1}{2}$ 
    and may proceed as follows: For $a \in [0,\frac{1}{4}]$ and $b \in [a,1]$ 
    define the transformation $T_{a,b} \colon \mathbb{I} \rightarrow [a,b]$ by $T_{a,b}(t) := a + (b-a)t$.
    Choose $n_0 \in \mathbb{N}$ sufficiently large so that $\gamma(\mathbb{I})(1-\frac{1}{n_0})^{d-1} > \frac{1}{2}$ 
    holds and set $a_n := \frac{1}{n}$ for every $n \geq n_0$. 
    Defining the function $\Phi_n: [a_n,1] \rightarrow [0,\infty)$ by 
    $$
    \Phi_n(s) := \int_\mathbb{I}(1-a_n - (s-a_n)t)^{d-1} \mathrm{d} \gamma(t) 
    = \int_{[a_n,s]}(1-t)^{d-1} \mathrm{d} \gamma^{T_{a_n,s}}(t)
    $$ 
    it follows that $\Phi_n$ is strictly decreasing and continuous on $[a_n,1]$ and that 
    $\Phi_n(a_n) \geq \frac{1}{2}$ as well as $\Phi_n(1) \leq \frac{1}{2}$ hold. 
    Hence there exists a unique $b_n \in [a_n,1]$ fulfilling $\Phi(b_n) = \frac{1}{2}$. 
    Letting $\xi_n$ denote the unique probability measure on $\mathcal{B}([0,\infty))$ 
    fulfilling $\xi_n([0,a_n)) = 0 = \xi_n((b_n,1])$, 
    coinciding with $\gamma^{T_{a_n,b_n}}$ on $\mathcal{B}([a_n,b_n])$ and with $\gamma$ on $\mathcal{B}((1,\infty))$
    we obviously have $\xi_n \in \mathcal{P}_{\mathcal{W}_d}$ for every $n \geq n_0$. 
    Considering $a_n \overset{n\rightarrow \infty}{\longrightarrow} 0$ it is straightforward to verify
     that $b_n \overset{n\rightarrow \infty}{\longrightarrow} 1$. 
     Hence, for every bounded continuous function $f \colon [0,\infty) \rightarrow \mathbb{R}$ 
     using change of coordinates it follows that   
    \begin{align*}
       \int_{[0,\infty)} f(t) \mathrm{d}\xi_n(t) &=   \int_{[0,a_n) \cup (b_n,1]} f(t) \mathrm{d}\xi_n(t) 
           +  \int_{[a_n,b_n]} f(t) \mathrm{d}\xi_n(t) +  \int_{(1,\infty)} f(t) \mathrm{d}\xi_n(t) \\
           &= 0 + \int_{[a_n,b_n]} f(t) \mathrm{d}\gamma^{T_{a_n,b_n}}(t) +  \int_{(1,\infty)} f(t) \mathrm{d}\xi_n(t) \\
           &=  \int_{\mathbb{I}} f \circ T_{a_n,b_n}(t) \mathrm{d}\gamma(t) +  \int_{(1,\infty)} f(t) \mathrm{d}\gamma(t)
    \end{align*}    
    and the latter sum converges to  $\int_{[0,\infty)} f(t) \mathrm{d}\gamma(t)$.
    In other words: the sequence $(\xi_n)_{n \in \mathbb{N}}$ converges weakly to $\gamma$. \\
    (ii) Suppose that $\gamma(\{0\}) = \frac{1}{2}$ holds, in which case we have $\gamma(\mathbb{I}) = \frac{1}{2}$.
    Letting $F_\gamma$ denote the distribution function of $\gamma$ and $F_\gamma^-$ its quasi-inverse 
    then we obviously have $F_\gamma^{-}(\frac{1}{2}+\frac{1}{n}) > 1$. 
    For every $n \geq 3$ define another distribution function $F_{\zeta_n}$ by
    $$
    F_{\zeta_n}(z) := \left(\frac{1}{2}+\frac{1}{n}\right)\mathbf{1}_{[1-\sqrt[d-1]{\frac{n}{n+2}},F_\gamma^{-}(\frac{1}{2}+\frac{1}{n}))}(z) + F_{\gamma}(z)\mathbf{1}_{[F_\gamma^{-}(\frac{1}{2}+\frac{1}{n}),\infty)}(z)
    $$
    for every $z \in [0,\infty)$. Then obviously the probability measures $\zeta_n$ induced by $F_{\zeta_n}$ is
    a $d$-Williamson measure for every $n \geq 3$. 
    If $z \in (0,\infty)$ fulfills $F_\gamma(z) = \frac{1}{2}$, then 
    $z < F_\gamma^{-}(\frac{1}{2} + \frac{1}{n})$, hence, using the facts that $\sqrt[d-1]{\frac{n}{n+2}} \overset{n \rightarrow \infty}{\longrightarrow} 1$ and $\frac{1}{2}+\frac{1}{n} \overset{n \rightarrow \infty}{\longrightarrow} \frac{1}{2}$ 
    shows convergence of $(F_{\zeta_n}(z))_{n \in \mathbb{N}}$ to $F_\gamma(z)$.  
    If $z \in (0,\infty)$ fulfills $F_\gamma(z) > \frac{1}{2}$, then we may 
    choose $K \in \mathbb{N}$ sufficiently large so that $\frac{1}{2} + \frac{1}{K} \leq F_\gamma(z)$ and 
    $F_\gamma^-(\frac{1}{2}+\frac{1}{K}) \leq z$ hold. 
    Then $F_{\zeta_n}(z) = F_\gamma(z)$ for all $n \geq K$, which completes the proof.  
\end{proof}
Building upon the previous lemma we prove that a typical element of $\mathcal{P}_{nor}^d$ is a Williamson measure.
\begin{Lemma}\label{lem:williamson_co_meager}
The family of $d$-Williamson-measures $\mathcal{P}_{\mathcal{W}_d}$ is co-meager in $(\mathcal{P}_{nor}^d,\tau_w)$.
\end{Lemma}
\begin{proof}
    It suffices to show that the set
    \begin{equation}\label{eq:positive_in_zero}
    N:=\{\gamma \in \mathcal{P}_{nor}^d \colon \gamma(\{0\}) >0\}
    \end{equation}
    is of first Baire category in $\mathcal{P}_{nor}^d$, which can be done as follows. 
    For every $k \geq 2$ defining the set $\mathcal{P}_k$ by
    $$
    \mathcal{P}_k := \left\{\gamma \in \mathcal{P}_{nor}^d \colon \gamma(\{0\}) \geq \frac{1}{k}\right\}
    $$
    we obviously have $N \subseteq \bigcup_{k=2}^\infty \mathcal{P}_k$. 
    Portmanteau's theorem implies that for every sequence $\gamma_1,\gamma_2, ... \in \mathcal{P}_k$ 
    converging weakly to some $\gamma \in \mathcal{P}_{nor}^d$ we have  that
    $\gamma(\{0\}) \geq \limsup_{n \rightarrow \infty} \gamma_n(\{0\}) \geq \frac{1}{k}$ holds, so 
    $\mathcal{P}_k$ is weakly closed in $\mathcal{P}_{nor}^d$. 
    Letting $\mathcal{O} \subseteq \mathcal{P}_{nor}^d$ denote a non-empty open set 
    then denseness of $\mathcal{P}_{\mathcal{W}_d}$ in $\mathcal{P}_{nor}^d$ yields the existence of a measure 
    $\tilde{\gamma} \in \mathcal{O} \cap \mathcal{P}_{\mathcal{W}_d}$ with $\tilde{\gamma} \notin \mathcal{P}_k$.
    This shows that $\mathcal{P}_k$ is nowhere dense in $\mathcal{P}_{nor}^d$, so the set $N$ is of first 
    Baire category, implying that $\mathcal{P}_{\mathcal{W}_d}$ is co-meager in $\mathcal{P}_{nor}^d$.
\end{proof}
In the following we extend the result from \cite{cat_exchange}, stating that a typical Archimedean copula is strict, 
to arbitrary dimensions $d \geq 2$. 
Building upon the tools developed in the previous sections and in \cite{mult_arch} the proof provided below 
is much simpler than the one for the bivariate setting established in \cite{cat_exchange}. 
In fact, we are even able to prove a stronger result: A typical multivariate Archimedian copula has full support.
As before we work in $\mathcal{P}_{\mathcal{W}_d}$ and then translate to $\mathcal{C}_{ar}^d$. 
\begin{Lemma}\label{thm:ful_sup_wil_meas_co-meager}
The set $\mathcal{P}_{\mathcal{W}_d}^{fs}$ is co-meager in $(\mathcal{P}_{\mathcal{W}_d},\tau_w)$.
\end{Lemma}
\begin{proof}
It suffices to show that  
$$
\mathcal{W} := \{\gamma \in \mathcal{P}_{\mathcal{W}_d} \colon \text{there exists an interval }(a,b)\subseteq [0,\infty) \text{ with } \gamma((a,b)) = 0\}
$$
is of first Baire category in $(\mathcal{P}_{\mathcal{W}_d},\tau_w)$. 
Let $q_1,q_2,\ldots$ be an enumeration of $(0,\infty) \cap \mathbb{Q}$, set $q_0 := 0$, and, 
for $(i,k) \in \mathbb{N}_0 \times \mathbb{N}$ define the sets $\mathcal{W}_{i,k}$ by
\begin{align*}
\mathcal{W}_{i,k} := \begin{cases}
        \{\gamma \in \mathcal{P}_{\mathcal{W}_d}\colon \gamma((q_i - \tfrac{1}{k}, q_i + \tfrac{1}{k})) = 0\} \text{ if }
             i \geq 1, \\
        \{\gamma \in \mathcal{P}_{\mathcal{W}_d}\colon \gamma((0, \tfrac{1}{k})) = 0\} \text{ if } i=0.
    \end{cases}
\end{align*}
Then we obviously have $\mathcal{W} \subseteq \bigcup_{i \in \mathbb{N}_0} \bigcup_{k \in \mathbb{N}}\mathcal{W}_{i,k}$. We show that for $(i,k) \in \mathbb{N}^2$ the set $\mathcal{W}_{i,k}$ is weakly closed. 
If $(\gamma_n)_{n \in \mathbb{N}}$ is a sequence in $\mathcal{W}_{i,k}$ converging weakly to some 
$\gamma \in \mathcal{P}_{\mathcal{W}_d}$, then Portmanteau's theorem shows that
$$
0 = \liminf_{\ell \rightarrow \infty}\gamma_\ell((q_i - \tfrac{1}{k}, q_i + \tfrac{1}{k})) \geq \gamma((q_i - \tfrac{1}{k}, q_i + \tfrac{1}{k})) \geq 0,
$$
so $\gamma \in \mathcal{W}_{i,k}$. Proceeding analogously yields that the sets $\mathcal{W}_{0,k}$ are 
weakly closed too. 
Considering that, according to Lemma \ref{lem:dense_discrete_measures}, the set of Williamson measures with full 
support is dense in $\mathcal{P}_{\mathcal{W}_d}$ it follows that each set $\mathcal{W}_{i,k}$ is nowhere dense 
in $\mathcal{P}_{\mathcal{W}_d}$, which implies that $\mathcal{W}$ is of first Baire category. 
Hence, by definition $\mathcal{P}_{\mathcal{W}_d}^{fs}$ is co-meager in $(\mathcal{P}_{\mathcal{W}_d},\tau_w)$.
\end{proof}
Translating from $ \mathcal{P}_{\mathcal{W}_d}$ to $\mathcal{C}_{ar}^d$ yields that a topologically 
typical Archimedean copula has full support - the following result holds: 
\begin{theorem}\label{thm:typical_full_supp}
The set $\{C \in \mathcal{C}_{ar}^d \colon \mathrm{supp}(\mu_{C}) = \mathbb{I}^d\}$ is co-meager in $(\mathcal{C}_{ar}^d,d_\infty)$.
\end{theorem}
\begin{proof}
    Immediate consequence of Lemma \ref{lem:hom_archimedean}, Theorem \ref{thm:arch_full_support} and Lemma \ref{thm:ful_sup_wil_meas_co-meager}.
\end{proof}
The subsequent corollary is now an immediate consequence: 
\begin{Cor}\label{cor:strict_wil_meas_co-meager}
The set $\mathcal{P}_{\mathcal{W}_d}^s$ is co-meager in $\mathcal{P}_{\mathcal{W}_d}$ w.r.t. the weak topology.
Moreover, $\mathcal{C}_{ar,s}^d$ is co-meager in $(\mathcal{C}_{ar}^d,d_\infty)$.
\end{Cor}
\begin{proof}
    The first assertion follows from the facts that
     $\mathcal{P}_{\mathcal{W}_d}^{fs} \subseteq \mathcal{P}_{\mathcal{W}_d}^s$ is co-meager 
     in $\mathcal{P}_{\mathcal{W}_d}$ and that super-sets of co-meager sets are themselves co-meager. 
     Having this, applying Lemma \ref{lem:hom_archimedean}, Theorem \ref{thm:arch_full_support} and 
     \cite[Lemma 5.5]{mult_arch} proves the second assertion. 
\end{proof}
As mentioned at the beginning of this section, since the space of $d$-dimensional Archimedean copulas is not 
complete, the fact that the families $\mathcal{C}_{ar,s}^d$ and $\{C \in \mathcal{C}_{ar}^d\colon \mathrm{supp}(\mu_C) =\mathbb{I}^d\}$ are of second Baire category is not obvious. 
Using the results established so far in this section we can, however, prove the following assertions.
\begin{Lemma}\label{lem:strict_second_category}
The sets $\mathcal{P}_{\mathcal{W}_d}^{fs}$ and  $\mathcal{P}_{\mathcal{W}_d}^s$ are of second Baire category in $(\mathcal{P}_{\mathcal{W}_d},\tau_w)$.
\end{Lemma}
\begin{proof}
Suppose that $\mathcal{P}_{\mathcal{W}_d}^{fs}$ were of first Baire category in $\mathcal{P}_{\mathcal{W}_d}$. 
Then, considering that $\mathcal{P}_{nor}^d$ contains $\mathcal{P}_{\mathcal{W}_d}$, it would be of first Baire 
category in $\mathcal{P}_{nor}^d$ as well. 
Applying Lemma \ref{lem:williamson_co_meager} and Theorem \ref{thm:ful_sup_wil_meas_co-meager} yields that 
$\mathcal{P}_{nor}^d\setminus \mathcal{P}_{\mathcal{W}_d}$ and $\mathcal{P}_{\mathcal{W}_d}\setminus \mathcal{P}_{\mathcal{W}_d}^{fs}$ are both of first Baire category in $\mathcal{P}_{nor}^d$. Therefore, 
writing $\mathcal{P}_{nor}^d = (\mathcal{P}_{nor}^d\setminus \mathcal{P}_{\mathcal{W}_d}) \cup (\mathcal{P}_{\mathcal{W}_d}^{fs} \cup (\mathcal{P}_{\mathcal{W}_d}\setminus \mathcal{P}_{\mathcal{W}_d}^{fs}))$ 
and using the fact that finite unions of sets of first Baire category are themselves of first Baire category 
would yield that $\mathcal{P}_{nor}^d$ is of first Baire category in itself. A contradiction. 
Therefore $\mathcal{P}_{\mathcal{W}_d}^{fs}$ is of second Baire category in $\mathcal{P}_{\mathcal{W}_d}$. 
Considering $\mathcal{P}_{\mathcal{W}_d}^{fs} \subseteq \mathcal{P}_{\mathcal{W}_d}^{s}$ it follows that 
$\mathcal{P}_{\mathcal{W}_d}^{s}$ is of second Baire category in $\mathcal{P}_{\mathcal{W}_d}$, which completes the 
proof.  
\end{proof}
Using Lemma \ref{lem:hom_archimedean} and translating the previous lemma to $(\mathcal{C}_{ar}^d$,$d_\infty)$
we have shown the following result: 
\begin{Cor}\label{cor:strict_second_category}
The sets $\mathcal{C}_{ar,s}^d$ and  $\{C \in \mathcal{C}_{ar}^d \colon \mathrm{supp}(\mu_{C}) = \mathbb{I}^d\}$ are of second Baire category in $(\mathcal{C}_{ar}^d$,$d_\infty)$.
\end{Cor}
As next step we will show that a typical $d$-dimensional Archimedean copula is not absolutely continuous 
and again start with proving the corresponding result for the family $\mathcal{P}_{\mathcal{W}_d}$.
\begin{Lemma}\label{lem:abs_cont_will_first_cat} 
The set $\mathcal{P}_{\mathcal{W}_d}^{abs}$ is of first Baire category in $(\mathcal{P}_{\mathcal{W}_d},\tau_w)$.
\end{Lemma}
\begin{proof}
    Suppose that $\gamma \in \mathcal{P}_{\mathcal{W}_d}^{abs}$ has density $\mathcal{k}_\gamma$; for 
    $n \in \mathbb{N}$ define the sets $M_n^\gamma \in \mathcal{B}([0,\infty))$ by
    $$
    M_n^\gamma := \{z \in [0,\infty)\colon \mathcal{k}_\gamma(z) > n\}
    $$
    and set
    $$
    \mathcal{M}_n := \{\gamma \in \mathcal{P}_{W_d}^{abs}\colon\gamma(M_\gamma^n)\leq \tfrac{1}{4}\}.
    $$
    Considering $\lambda(\bigcap_{n = 1}^\infty M_n^\gamma) = 0$, absolute continuity of $\gamma$ 
    yields $\gamma(\bigcap_{n = 1}^\infty M_n^\gamma) = 0$, so for sufficiently large $n$ we have
     $\gamma(M_n^\gamma) < \frac{1}{4}$. This shows that 
    $$
    \mathcal{P}_{\mathcal{W}_d}^{abs} \subseteq \bigcup_{n \in \mathbb{N}}\mathcal{M}_n,
    $$
    it hence suffices to show that $\mathcal{M}_n$ is nowhere dense in $(\mathcal{P}_{\mathcal{W}_d},\tau_w)$,
    which can be done as follows: 
    Let $\beta \in \mathcal{P}_{\mathcal{W}_d}^{dis}$  be an arbitrary discrete $d$-Williamson measure 
    with only finitely many point masses, i.e.,
    $
    \beta := \sum_{i=1}^N\alpha_i\delta_{x_i},
    $
    whereby $0<x_1 < x_2 < \cdots < x_N < \infty$ and $\alpha_1,...,\alpha_N \in (0,1]$ fulfill
     $\sum_{i=1}^N\alpha_i = 1$. 
    We will show that it is not possible to construct a sequence in $\mathcal{M}_n$ converging weakly to 
    $\beta$. Set $x_0 = 0$ and define $\rho$ by
    $$
    \rho := \tfrac{1}{8nN}\min\big\{\min\{|x_i - x_j|\colon i,j \in \{0,...,N\}\},1\big\} > 0.
    $$
    Then obviously we have $\lambda(\bigcup_{i=1}^N(x_i - \rho, x_i + \rho)) \leq \frac{1}{4n}$.
    Letting $f \colon [0,\infty) \rightarrow \mathbb{R}$ denote a continuous function fulfilling 
    $f(x_i) = 1$ and $f\mid_{(x_i-\rho,x_i+\rho)} \in (0,1]$ for every $i \in \{1,\ldots,N\}$, and being 
    identical to $0$ on $[0,\infty) \setminus \bigcup_{i=1}^N (x_i - \rho, x_i + \rho)$. 
    By construction, on the one hand we have $\int_{[0,\infty)}f\mathrm{d}\beta = 1$. 
    On the other hand, for arbitrary but fixed $\gamma \in \mathcal{M}_n$ it follows that
    \begin{align*}
    \int_{[0,\infty)} f \mathrm{d}\gamma = \int_{M_n^\gamma} f \mathrm{d}\gamma + \int_{[0,\infty)\setminus M_n^\gamma} f \mathrm{d}\gamma \leq \frac{1}{4} + \int_{[0,\infty)\setminus M_n^\gamma} f \mathrm{d}\gamma \leq \frac{1}{2} < 1 =\int_\mathbb{I} f \mathrm{d}\beta.
    \end{align*}
    Finally, using the fact that discrete Williamson-measures with finitely many point masses are dense in
     $(\mathcal{P}_{\mathcal{W}_d},\tau_w)$ (see \cite{mult_arch}) it follows that $\mathcal{M}_n$ is nowhere dense in $\mathcal{P}_{\mathcal{W}_d}$, which proves the desired result.
\end{proof}
\begin{theorem}\label{lem:abs_cont_first_cat} 
The family $\mathcal{C}_{ar,abs}^d$ is of first Baire category in $(\mathcal{C}_{ar}^d,d_\infty)$.
\end{theorem}
\begin{proof}
     Immediate consequence of the fact that homeomorphisms map meager sets to meager sets, Lemma \ref{lem:hom_archimedean}, Corollary \ref{cor:regularity_arch} and Lemma \ref{lem:abs_cont_will_first_cat}.
\end{proof}
We now turn towards Williamson measures having non-degenerated discrete component and show that 
these elements are topologically atypical. 
\begin{Lemma}\label{lem:arch_no_atoms}
The set
$$\{\gamma \in \mathcal{P}_{\mathcal{W}_d}\colon \gamma \text{ has no point masses}\}$$
is co-meager in $(\mathcal{P}_{\mathcal{W}_d},\tau_w)$.
\end{Lemma}
\begin{proof}
    It suffices to show that the set
    $$
    \mathcal{P}_{\mathcal{W}_d}^{p} := \{\gamma \in \mathcal{P}_{\mathcal{W}_d}\colon \exists z \in (0,\infty)
     \text{ with } \gamma(\{z\}) > 0\}
    $$
    is of first Baire category in $ \mathcal{P}_{\mathcal{W}_d}$. 
    For arbitrary $k \in \mathbb{N}$ and $2 \leq m \in \mathbb{N}$ defining the set $ \mathcal{W}_{k,m}$ by
    $$
    \mathcal{W}_{k,m} := \left\{\gamma \in \mathcal{P}_{\mathcal{W}_d} \colon \exists z \in \left[\frac{1}{m},m\right] \text{ such that } \gamma(\{z\})\geq \frac{1}{k}\right\}
    $$
    we obviously have $\mathcal{P}_{\mathcal{W}_d}^p\subseteq \bigcup_{k = 1}^\infty\bigcup_{m = 2}^\infty\mathcal{W}_{k,m}$. 
 We will show that $\mathcal{W}_{k,m}$ is closed in $(\mathcal{P}_{\mathcal{W}_d},\tau_w)$ 
 and proceed as follows: Suppose that $(\gamma_n)_{n \in \mathbb{N}}$ is a sequence in $\mathcal{W}_{k,m}$ 
 converging weakly to some $\gamma \in \mathcal{P}_{\mathcal{W}_d}$. 
 Then there exists some sequence $(z_n)_{n \in \mathbb{N}}$ in $\left[\frac{1}{m},m\right]$ 
 fulfilling $\gamma_n(\{z_n\}) \geq \frac{1}{k}$ for every $n \in \mathbb{N}$. 
 Using compactness of $\left[\frac{1}{m},m\right]$ there exists some subsequence 
 $(z_{n_j})_{j \in \mathbb{N}}$ with limit $z^* \in \left[\frac{1}{m},m\right]$. 
 We want to prove that $\gamma(\{z^*\}) \geq \frac{1}{k}$ holds. To simplify notation let 
 $F_{n_j}$ and $F_\gamma$ denote the distribution functions of $\gamma_{n_j}$ and $\gamma$, respectively. 
 Choose an arbitrarily small $h > 0$ such that $z^* + h$ and $z^* - h$ are continuity points of $F_\gamma$. 
 Then using weak convergence of $(\gamma_n)_{n \in \mathbb{N}}$ to $\gamma$ yields that
    $$
  \lim_{j \rightarrow \infty} F_{n_j}(z^* + h) = F_\gamma(z^* + h) \text{ and } 
  \lim_{j \rightarrow \infty} F_{n_j}(z^* - h) = F_\gamma(z^* - h).
    $$
    Furthermore, convergence of $(z_{n_j})$ to $z^*$ implies the existence of some 
    $j_0 \in \mathbb{N}$ fulfilling that $|z_{n_j} - z^*| < h$ holds for all $j \geq j_0$. 
    Finally, using monotonicity of distribution functions it follows that
    $$
    \frac{1}{k} \leq \gamma_{n_j}(\{z_{n_j}\}) = F_{n_j}(z_{n_j}) - F_{n_j}(z_{n_j}-) 
    \leq F_{n_j}(z^* + h) - F_{n_j}(z^* - h)
    $$
    holds for sufficiently large $j$, which implies
    $$
    \frac{1}{k} \leq F_\gamma(z^* + h) - F_\gamma(z^* - h).
    $$
    Since $h>0$ can be chosen arbitrarily small we finally obtain that $\gamma(\left\{z^*\right\}) \geq \frac{1}{k}$, i.e., $\mathcal{W}_{k,m}$ is closed in $(\mathcal{P}_{\mathcal{W}_d},\tau_w)$. 
    According to Lemma \ref{lem:dense_discrete_measures} the family of absolutely continuous Williamson measures 
    is dense in $(\mathcal{P}_{\mathcal{W}_d},\tau_w)$, hence $\mathcal{W}_{k,m}$ is nowhere dense in 
    $(\mathcal{P}_{\mathcal{W}_d},\tau_w)$. 
    This shows that $\mathcal{P}_{\mathcal{W}_d}^{p}$ is of first Baire category in $\mathcal{P}_{\mathcal{W}_d}$
    and the proof is complete.
\end{proof}
Translating the previous lemma to $\mathcal{C}_{ar}^d$ and using  Lemma \ref{lem:hom_archimedean}, Lemma \ref{lem:arch_no_atoms}, Theorem \ref{thm:main_regularity_result} we have shown that typical Archimedean copulas 
have degenerated discrete component - the following result holds:
\begin{theorem}\label{thm:typical_archimedean}
The set
$$\{C\in \mathcal{C}_{ar}^d \colon \mu_C^{dis}(\mathbb{I}^d) = 0\}$$
is co-meager in $(\mathcal{C}_{ar}^d,d_\infty)$.
\end{theorem}
Proceeding analogously to the proof of Corollary \ref{cor:strict_second_category} yields the following: 
\begin{Cor}
The set
$$\{C\in \mathcal{C}_{ar}^d \colon \mu_C^{dis}(\mathbb{I}^d) = 0\}$$
is of second Baire category in $(\mathcal{C}_{ar}^d,d_\infty)$.
\end{Cor}
\begin{Rem}
\normalfont
    Considering Theorem \ref{thm:equiv_discrete_path} and Theorem \ref{thm:typical_archimedean} 
    we conclude that topologically typical Archimedean copulas do not exhibit pathological behavior, i.e., 
    a typical $d$-dimensional Archimedean copula $C$ fulfills $C \not\in \mathcal{C}_{ar,p}^d$.
\end{Rem}
Combining Theorem \ref{thm:typical_full_supp}, Theorem \ref{lem:abs_cont_first_cat} and Theorem \ref{thm:typical_archimedean} yields the following surprising 
main result on topologically typical multivariate Archimedean copulas:
\begin{Cor}\label{cor:typical_arch_copula}
    A topologically typical $d$-dimensional Archimedean copula $C$ has full support, has degenerated discrete 
    component and is not absolutely continuous.
\end{Cor}
We conclude this section with an example of a topologically atypical bivariate Archimedean copula.
\begin{Ex}\label{ex:arch_abs_and_ex}
\normalfont
    Consider the $2$-Williamson measure $\gamma \in \mathcal{P}_{\mathcal{W}_2}$, defined via its distribution function
     by
    \begin{align*}
    \gamma([0,z]) := \begin{cases}
        0,& \text{ if } z \in [0,\frac{1}{4}),\\
        \frac{2}{3},& \text{ if } z \in [\frac{1}{4},1),\\
        \frac{1}{3}(\frac{1}{4}\sqrt{z-1}+2),& \text{ if } z \in [1,2),\\
        \frac{1}{8}(\sqrt[3]{z-2}+7),& \text{ if } z \in [2,3),\\
        1,& \text{ if } z \geq 3,
    \end{cases}
    \end{align*}
    for every $z \in [0,\infty)$. 
    Obviously (see Figure \ref{fig:dis_abs_cont_arch_meas_gen}) $\gamma$ has a non-degenerated discrete 
    as well as a non-degenerated absolutely continuous component. 
    Straightforward calculations show that the corresponding generator $\psi$ is given by
    $$
    \psi(z) = \begin{cases}
        1-\frac{233 z}{288},& \text{ if } z \in [0,\frac{1}{3}),\\
        \frac{27\sqrt[3]{(1-2z)^4}-4\sqrt[3]{z}(38z-63)}{288\sqrt[3]{z}},& \text{ if } z \in (\frac{1}{3},\frac{1}{2}],\\
        \frac{-3\sqrt{z^3} + \sqrt{(1-z)^3} + 12\sqrt{z}}{18\sqrt{z}},& \text{ if } z \in (\frac{1}{2},1],\\
        \frac{2}{3}(1-\frac{z}{4}),& \text{ if } z \in (1,4],\\
        0,& \text{ if } z > 4.
    \end{cases}
    $$
    The induced Archimedean copula $C_\psi$ has non-degenerated discrete component and is non-strict, 
    so according to Corollary \ref{cor:typical_arch_copula} the copula $C_\psi$ is atypical. 
    A sample of the copula $C_\psi$ is depicted in Figure \ref{fig:dis_abs_cont_arch}.
\end{Ex}
\begin{figure}[!ht]
	\centering
	\includegraphics[width=1\textwidth]{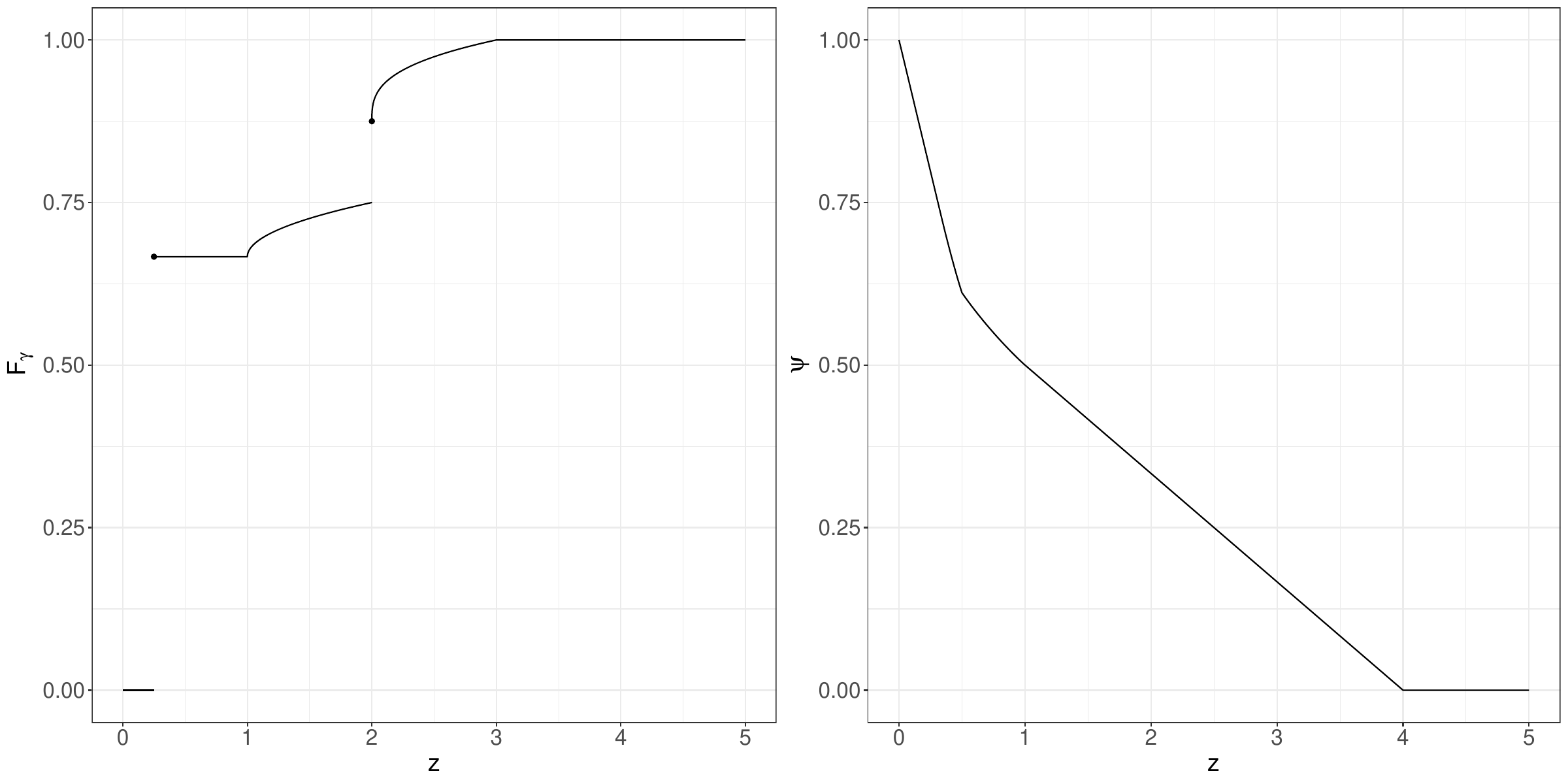}
	\caption{The distribution function $F_\gamma$ (left panel) and the associated generator $\psi$ (right panel) 
	considered in Example \ref{ex:arch_abs_and_ex}.}
\label{fig:dis_abs_cont_arch_meas_gen}
\end{figure}
\begin{figure}[!ht]
	\centering
	\includegraphics[width=1\textwidth]{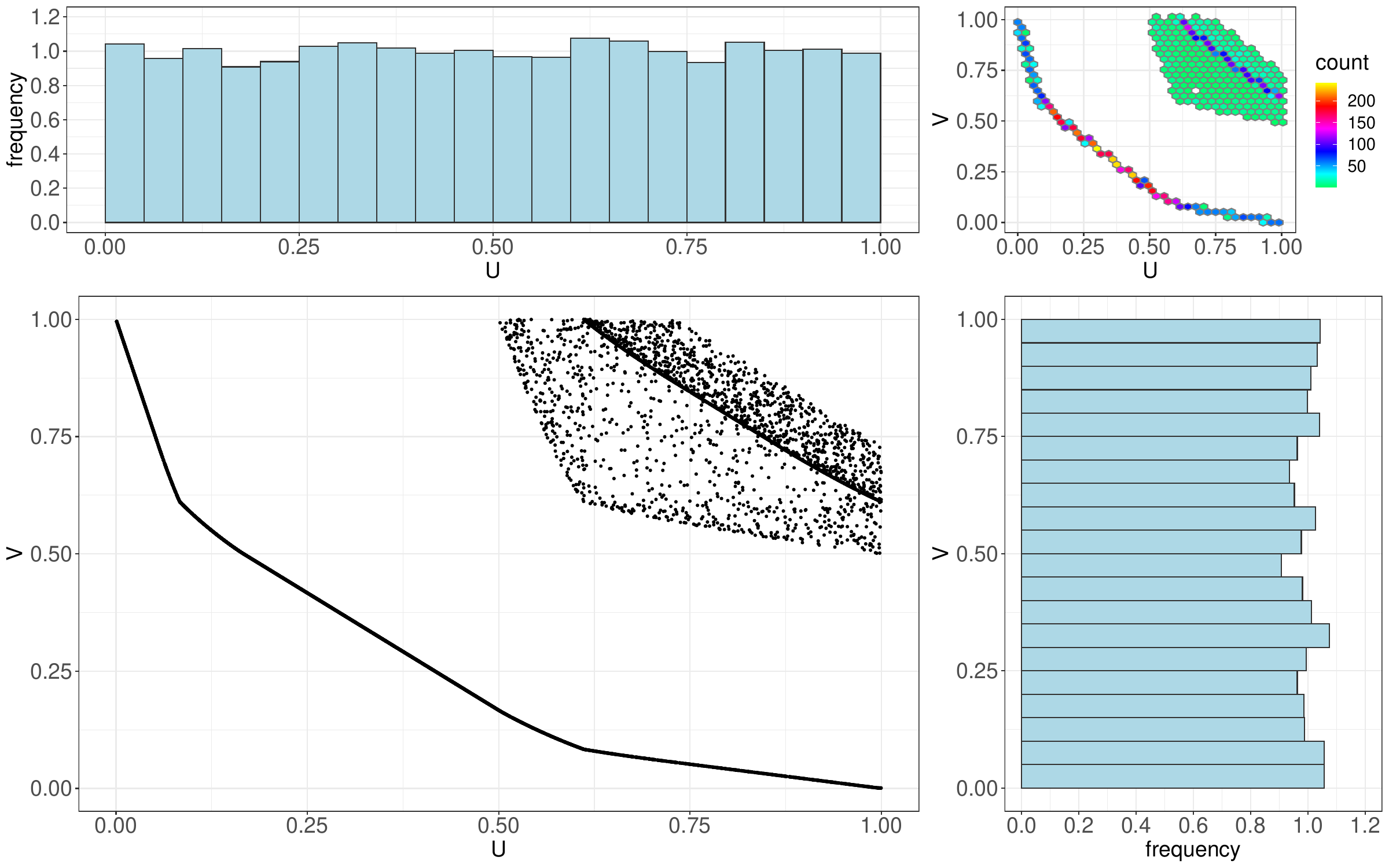}
	\caption{Sample of size 10000 of the Archimedean copula $C_\psi$ with $\psi$ being the generator from 
	Example \ref{ex:arch_abs_and_ex}, its histogram and the two marginal histograms; 
	sample generated via conditional inverse sampling.}
\label{fig:dis_abs_cont_arch}
\end{figure}
\textbf{Acknowledgements.}
The first author gratefully acknowledges the support of Red Bull GmbH within the ongoing Data Science collaboration with the university of Salzburg. The second author gratefully acknowledges the support of the WISS 2025 project ‘IDA-lab Salzburg’
(20204-WISS/225/197-2019 and 20102-F1901166-KZP).

\section*{Declaration}
\textbf{Conflict of interest.} The authors have no competing interests to declare that are relevant to the content of this
article.
\bibliographystyle{plainnat}

\begin{thebibliography}{99}

\bibitem[Billingsley, 1999]{Bill}
Billingsley, P.: Convergence of Probability Measures, Second edition. John Wiley \& Sons, New York (1999)

\bibitem[Bruckner and Thomson, 1997]{bruck1997}
Bruckner, A.M., Thomson, B.S.:
\newblock
Real Analysis.
\newblock
Prentice-Hall, London (1997)


\bibitem[Dietrich and Trutschnig, 2025]{dietrich2024}
Dietrich, N., Trutschnig, W.:
\newblock
On differentiability and mass distributions of typical bivariate copulas.
\newblock
Fuzzy Sets Syst. 498, 109150 (2025). 
\newblock \url{https://doi.org/10.1016/j.fss.2024.109150}
%\doi{10.1016/j.fss.2024.109150}.

\bibitem[Durante et al., 2022]{durante2022}
Durante, F., Fernández-Sánchez, J., Ignazzi, C.:
\newblock Baire category results for stochastic orders.
\newblock Rev. Real Acad. Cienc. Exactas Fis. Nat. - A: Mat. 116, 188 (2022).
\newblock \url{https://doi.org/10.1007/s13398-022-01324-3}


\bibitem[Durante et al., 2016a]{cat_exchange}
Durante, F. Fernández-Sánchez, J., Trutschnig,  W.:
\newblock Baire category results for exchangeable copulas.
\newblock Fuzzy Sets Syst. 284, 146-151 (2016).
\newblock \url{https://doi.org/10.1016/j.fss.2015.04.010}

\bibitem[Durante et al., 2016b]{cat_quas_cop}
Durante, F., Fernández-Sánchez, J., Trutschnig, W.:
\newblock Baire category results for quasi–copulas.
\newblock Depend. Model. 4, 215-223 (2016).
\newblock \url{https://doi.org/10.1515/demo-2016-0012}

\bibitem[Durante et al., 2007]{durante2007}
Durante, F., Quesada-Molina, J.J., Sempi, C.:
\newblock A Generalization of the Archimedean Class of Bivariate Copulas.
\newblock Ann. I. Stat. Math. 59, 487-498 (2007).
\newblock \url{https://doi.org/10.1007/s10463-006-0061-9}

\bibitem[Durante and Sempi, 2015]{dur_princ}
Durante, F., Sempi, C.:
\newblock Principles of Copula Theory.
\newblock Chapman and Hall/CRC (2015)


    \bibitem[Fernández Sánchez and Trutschnig, 2015a]{typ_cop_sing}
    Fernández Sánchez, J., Trutschnig, W.:
    \newblock A typical copula is singular.
    \newblock 	J. Math. Anal. Appl. 430, 517-527 (2015).
    \newblock \url{https://doi.org/10.1016/j.jmaa.2015.05.009}


    \bibitem[Fernández Sánchez and Trutschnig, 2015b]{sing_arch}
    Fernández Sánchez, J., Trutschnig, W.:
    \newblock Singularity aspects of Archimedean copulas.
    \newblock 	J. Math. Anal. Appl. 432, 103-113 (2015).
    \newblock \url{https://doi.org/10.1016/j.jmaa.2015.06.036}



\bibitem[Hajłasz, 1993]{hajlasz1993}
Hajłasz, P.:
\newblock Change of variables formula under minimal assumptions.
\newblock Colloq. Math. 64, 93-101 (1993)


\bibitem[Kallenberg, 2002]{Kallenberg}
Kallenberg, O.:
\newblock Foundations of modern probability.
\newblock Springer-Verlag, New York (2002)


\bibitem[Kannan and Krueger, 1996]{kannan1996}
Kannan, R., Krueger, C.K.:
\newblock Advanced analysis on the real line.
\newblock Springer Verlage Publications (1996)

\bibitem[Kasper, 2024]{kasper2024}
Kasper, T.:
\newblock On convergence and singularity of conditional copulas of multivariate Archimedean copulas, and conditional dependence.
\newblock J. Multivar. Anal. 201, 105275 (2024).
\newblock \url{https://doi.org/10.1016/j.jmva.2023.105275}


\bibitem[Kasper et al., 2024]{mult_arch}
Kasper, T., Dietrich, N., Trutschnig, W.:
\newblock On convergence and mass distributions of multivariate Archimedean copulas and their interplay with the Williamson transform.
\newblock J. Math. Anal. Appl. 529, 127555 (2024).
\newblock \url{https://doi.org/10.1016/j.jmaa.2023.127555}

\bibitem[Kasper et al., 2021]{bernoulli}
Kasper, T., Fuchs, S., Trutschnig, W.:
\newblock On weak conditional convergence of bivariate Archimedean and Extreme Value copulas, and consequences to nonparametric estimation.
\newblock Bernoulli, 4, 2217 - 2240 (2021).
\newblock \url{https://doi.org/10.3150/20-BEJ1306}

\bibitem[Kim, 1996]{kim}
Kim, C.:
\newblock Uniform approximation of doubly stochastic operators.
\newblock Pac. J. Math., 26, 515 - 527 (1968)

\bibitem[Klenke, 2008]{Klenke}
Klenke, A.:
\newblock Wahrscheinlichkeitstheorie.
\newblock Springer Lehrbuch Masterclass Series, Berlin Heidelberg (2008)

\bibitem[Lange, 1973]{Lange}
Lange, K.:
\newblock Decompositions of Substochastic Transition Functions.
\newblock  	Proc. Am. Math. Soc. 37, 575 – 580 (1973).
\newblock \url{https://doi.org/10.2307/2039488}

\bibitem[Leoni, 2024]{leoni2024}
Leoni, G.:
\newblock A First Course in Sobolev Spaces, Graduate Studies in Mathematics, second edition, American Mathematical Soc.
(2024)


\bibitem[McNeil and Nešlehová, 2009]{neslehova}
McNeil, A., Nešlehová, J.:
\newblock 
Multivariate Archimedean Copulas, d-Monotone Functions and $\ell_1$-Norm Symmetric Distributions.
\newblock
Ann. Stat. 37, 3059 - 3097 (2009).
\newblock \url{ https://doi.org/10.1214/07-AOS556}


\bibitem[Nelsen, 2006]{nelsen2006}
 Nelsen, R.B.:
 \newblock An introduction to copulas. Springer Series in Statistics.
 \newblock
 Springer, New York, second edition (2006)


\bibitem[Oxtoby, 1980]{oxtoby1980}
Oxtoby, J.C.:
\newblock
Measure and category. A survey of the analogies between topological and measure spaces.
\newblock
Springer, second edition (1980)

\bibitem[Pap, 2002]{pap2002}
Pap, E.:
\newblock 
Handbook of Measure Theory.
\newblock
Elsevier, Amsterdam, Netherlands, Volume 1, first edition (2002)


\bibitem[Parthasarathy, 1967]{Parthasarathy}
Parthasarathy, K.R.:
\newblock 
Probability Measures on Metric Spaces.
\newblock
In: Probability and Mathematical Statistics: A Series of Monographs and Textbooks, pp.25-55.
\newblock
Academic Press (1967)

\bibitem[Pollard, 2001]{Pollard}
Pollard, D.:
\newblock
A User’s Guide to Measure Theoretic Probability. 
\newblock
Cambridge Series in
Statistical and Probabilistic Mathematics, Cambridge University Press, (2001)


\bibitem[Ricci, 2022]{diff_char}
Ricci, R.:
\newblock A differential characterization of the d-increasingness property.
\newblock Fuzzy Sets Syst. 433, 79-95 (2022).
\newblock \url{https://doi.org/10.1016/j.fss.2021.04.011}

\bibitem[Vaart, 1998]{van_d_vaart_asymp_stat}
Vaart, A.W.v.d.:
\newblock Asymptotic Statistics.
\newblock Cambridge University Press (1998)
 
\end{thebibliography}

\begin{appendices}
\section{Denseness of Williamson measures with full support}
In order to show that fully supported absolutely continuous, discrete and singular Williamson measures are dense in $(\mathcal{P}_{\mathcal{W}_d},\tau_w)$ we make use of the subsequent lemma. 
\begin{Lemma}\label{lem:dense_discrete_measures}
The following assertions hold:
\begin{itemize}
\item[(i)] The family of discrete $d$-Williamson measures with full support is dense in $(\mathcal{P}_{\mathcal{W}_d},\tau_w)$.
\item[(ii)] The family of singular $d$-Williamson measures with full support is dense in $(\mathcal{P}_{\mathcal{W}_d},\tau_w)$.
\item[(iii)] The family of absolutely continuous $d$-Williamson measures with full support is dense in $(\mathcal{P}_{\mathcal{W}_d},\tau_w)$.
\end{itemize}
\end{Lemma}
\begin{proof}
Fix $\gamma \in \mathcal{P}_{\mathcal{W}_d}$. Then according to 
\cite[Theorem Appendix B.2]{mult_arch} there exists a sequence 
$(\gamma_n)_{n \in \mathbb{N}}$ of absolutely continuous/discrete/singular Williamson measures 
converging weakly to $\gamma$. 
Considering an absolutely continuous/discrete/singular Williamson measure $\beta$ with full support (see, e.g.,  \cite[Theorem 6.1]{mult_arch} and \cite[Theorem 6.2]{mult_arch} for examples of such measures) 
and setting $\beta_n := (1- \frac{1}{n})\gamma_n + \frac{1}{n}\beta$ for every $n \in \mathbb{N}$ 
yields a sequence $(\beta_n)_{n \in \mathbb{N}}$ of absolutely continuous/discrete/singular Williamson 
measures with full support which converges weakly to $\gamma$.
\end{proof}
%Fix $C \in \mathcal{C}_{ar}^d$ and let $\gamma$ and $\psi$ be its corresponding Williamson measure and generator, respectively. Following \cite[Lemma 5.2]{mult_arch}, we obtain an explicit representation of the distribution function of $\gamma$ in terms of the generator $\psi$. The next Lemma is added for the sake of readability.
%\begin{Lemma}
%    Let $C \in \mathcal{C}_{ar}^d$, $\psi$ be its generator and $\gamma \in \mathcal{P}_{\mathcal{W}_d}$ its corresponding Williamson measure. Then
%\begin{equation}\label{eq:repres_williamson_measure}
%        \gamma([0,z]) = \sum_{k = 0}^{d-2}\frac{(-1)^k\psi^{(k)}(\frac{1}{z})}{k!}\frac{1}{z^k} + \frac{(-1)^{d-1}D^-\psi^{(d-2)}(\frac{1}{z})}{(d-1)!}\frac{1}{z^{d-1}}
%\end{equation}
%    holds for every $z>0$.
%\end{Lemma}
\end{appendices}
\end{document}